\documentclass[10pt, a4paper]{amsart}
\usepackage[utf8]{inputenc}
\usepackage{amsfonts}
\usepackage{amssymb}
\usepackage{amsmath}
\usepackage{amsthm}
\usepackage{newlfont}
\usepackage[mathcal]{euscript}
\usepackage{tikz}
\usepackage{enumerate}
\usetikzlibrary{shapes.geometric}
\usetikzlibrary{decorations.pathreplacing}

 \newtheorem{question}{Question}

 \newtheorem{thm}{Theorem}

\newtheorem{lem}[thm]{Lemma}

\newtheorem{definition}[thm]{Definition}
\newtheorem{lemma}[thm]{Lemma}
\newtheorem{proposition}[thm]{Lemma}
\newtheorem{corollary}[thm]{Corollary}

\newtheorem{remark}[thm]{Remark}

\DeclareMathOperator{\Int}{int}
\DeclareMathOperator{\diam}{diam}

\newcommand{\alp}{3.141592653589793238462643383279502884197169399375105820974944592307816406286-3.12}

\newcommand{\degr}{\mathrm{deg}}

\newcommand{\mo}{\mathrm{mod}}
\newcommand{\comp}{\mathrm{Comp}}

\newcommand{\id}{\mathrm{id}}

 \newcommand{\eps}{\varepsilon}
 
 \newcommand{\ta}{\theta}
 \newcommand{\Te}{\Theta}

 \def\B{{\mathcal B}}

 \newcommand{\Q}{\mathbb{Q}}

 \newcommand{\PQ}{\mathcal{Q}}
 \newcommand{\PP}{\mathcal{P}}

 \newcommand{\F}{\mathcal{F}}

 \newcommand{\N}{\mathbb{N}}
 \newcommand{\Z}{\mathbb{Z}}

 \newcommand{\R}{\mathbb{R}}
 \newcommand{\C}{\mathbb{C}}
 \newcommand{\Ci}{\mathbb{S}^1}

 \newcommand{\set}[1]{\left\{#1\right\}}

\begin{document}

\title[S-Limit shadowing for
circle maps]{S-Limit shadowing is generic for continuous Lebesgue measure preserving circle maps}
\author{Jozef Bobok}
\author{Jernej \v Cin\v c}
\author{Piotr Oprocha}
\author{Serge Troubetzkoy}

\address[J.\ Bobok]{Department of Mathematics of FCE, Czech Technical University in Prague, 
Th\'akurova 7, 166 29 Prague 6, Czech Republic}
\email{jozef.bobok@cvut.cz}

\address[J.\ \v{C}in\v{c}]{Faculty of Mathematics, University of Vienna, Oskar-Morgenstern-Platz 1,
A-1090 Vienna, Austria. -- and -- Centre of Excellence IT4Innovations - Institute for Research and Applications of Fuzzy Modeling, University of Ostrava, 30. dubna 22, 701 03 Ostrava 1, Czech Republic}
\email{jernej.cinc@osu.cz}

\address[P.\ Oprocha]{AGH University of Science and Technology, Faculty of Applied Mathematics, 
al.\ Mickiewicza 30, 30-059 Krak\'ow, Poland. -- and -- 
Centre of Excellence IT4Innovations - Institute for Research and Applications of Fuzzy Modeling, University of Ostrava, 30. dubna 22, 701 03 Ostrava 1, Czech Republic}
\email{oprocha@agh.edu.pl}

\address[S.\ Troubetzkoy]{Aix Marseille Univ, CNRS, Centrale Marseille, I2M, Marseille, France
postal address: I2M, Luminy, Case 907, F-13288 Marseille Cedex 9, France}
\email{serge.troubetzkoy@univ-amu.fr}

\date{\today}


\begin{abstract}
In this paper we show that generic continuous Lebesgue measure preserving circle maps have the s-limit shadowing property. In addition we obtain that s-limit shadowing is a generic property also for continuous circle maps.  In particular, this implies that classical shadowing, periodic shadowing and limit shadowing are generic in these two settings as well.
\end{abstract}

\maketitle

\section{Introduction}
The notion of {\em shadowing} (or pseudo orbit tracing property, see Definition~\ref{def:shadowing}) is a classical notion in the theory of dynamical systems. It was defined as a tool for better understanding of asymptotic aspects of diffeomorphisms dynamics independently by Anosov \cite{Anosov} and Bowen \cite{Bowen}. 
Informally, the shadowing property ensures that computational errors do not accumulate in the following sense: in the systems with shadowing property the approximate trajectories will reflect real dynamics up to some small error that is made in each iteration. In particular, this is of great importance in systems with sensitive dependence on initial conditions, where small errors may potentially result in large divergence of trajectories.

While we are still lacking the full classification of systems with shadowing, there are classes where its occurrence has been completely characterized. To look only at the most general results, all uniformly hyperbolic systems have shadowing property and Walters characterized the symbolic dynamical systems with shadowing property, they are shifts of finite type  (see the books \cite{Pilyugin,Pilyugin2} for more explanation). A useful collection of conditions characterizing shadowing in the latter setting was recently provided by Good and Meddaugh \cite{GM}. 

We call a property {\em generic} if it is satisfied on at least a dense $G_{\delta}$  of the underlying space.
A naturally related question which attracted attention of many researchers is the genericity of shadowing in dynamical systems. Hyperbolic systems are known to be rather special, and finding an answer in other classes of functions usually turns out to be a delicate matter. The first results in this direction were obtained in dimension one by Yano in \cite{Yano} for the space of homeomorphisms on the unit circle and Odani in \cite{Odani} for all smooth manifolds of the dimension at most $3$. A particularly nice technique was introduced by Pilyugin and Plamenevskaya in \cite{PilPla} who proved genericity of shadowing for homeomorphisms on any smooth compact manifold without boundary. This result was later extended using topological tools to a wider context (e.g. see \cite{PK}). 

Mizera proved  \cite{Mizera} that shadowing is a generic property in the class of continuous maps of the interval or circle. Recently, these results were extended to many other one-dimensional spaces, see \cite{KO,KMO,Med}. It turned out that non-invertibility is not an obstacle to obtain genericity of shadowing also in higher dimension \cite{PilOpr}.

In the literature there are many different generalizations of the shadowing property. Among the most natural ones is the {\em limit shadowing property} (see Definition~\ref{def:shadowing}), which was introduced by Pilyugin at al. \cite{Eriola}.
In this definition, error in consecutive elements of pseudo-trajectories tends to zero (so-called asymptotic pseudo orbit), but we require that accuracy of tracing increases with time.
While limit shadowing seems completely different than shadowing, it was proved in \cite{Limit} that transitive maps with limit shadowing also have the shadowing property.
Recently \cite{ACC}, it was proven that structurally stable diffeomorphisms and some pseudo-Anosov diffeomorphisms of the two-sphere satisfy both the shadowing and the limit shadowing property. 

In general, it can happen that for an asymptotic pseudo orbit which is also a $\delta$-pseudo orbit, the point which $\eps$-traces it and the point which traces it in the limit are two different points \cite{Barwel}. This shows that possessing a common point for such a tracing is a stronger property than the shadowing and limit shadowing properties together. The described property was introduced in \cite{Sakai} and is called {\em s-limit shadowing property} (again see Definition~\ref{def:shadowing} for the precise definition).
Not much is known about s-limit shadowing or even limit shadowing with respect to genericity in particular classes of functions. Besides the results mentioned above, the only result known to the authors which barely touches this problem is \cite{MazOpr}, where it is proven that in the class of continuous maps on manifolds of dimension $m\geq 1$, s-limit shadowing is a dense property with respect to the metric of uniform convergence.

The main difficulty in proving denseness or genericity of s-limit shadowing is its ``instability"; meaning that, intuitively, arbitrarily small perturbations can destroy it. Therefore, even the density result in \cite{MazOpr} relies on a very careful control of consecutive perturbations.
Our main theorem here, in particular, addresses the following very general question from \cite{MazOpr}:
\begin{center}
	(Q2) \textit{Is s-limit shadowing a $\mathcal{C}^0$-generic property on spaces where shadowing is generic?}
\end{center}

Let $\lambda$ denote the normalized Lebesgue measure on $I := [0,1]$ and let $\tilde\lambda$ denote the normalized Lebesgue measure on $\Ci$.
The particular setting that we are interested in this paper is the family of continuous Lebesgue measure preserving maps of the unit circle $C_{\tilde\lambda}(\mathbb{S}^1)$ endowed with topology of uniform convergence, which makes it a complete space. Topological and measure theoretical properties of generic Lebesgue measure preserving interval maps were studied in \cite{Bo91,BT,BCOT}. We obtain the following two new results:

\begin{thm}\label{thm:main}
	The s-limit shadowing property is generic in $C_{\tilde\lambda}(\mathbb{S}^1)$.
\end{thm}

\begin{corollary}
     The limit shadowing, periodic shadowing and shadowing property are generic in $C_{\tilde\lambda}(\mathbb{S}^1)$.
\end{corollary}

In the context of Lebesgue measure preserving functions, the genericity of shadowing was recently proven in \cite{GT} for homeomorphisms on manifolds (with or without boundary) of dimension at least $2$ where the authors use Oxtoby-Ulam's theorem \cite{OU} and its underlying subdivision of any such manifold. For manifolds of dimension $1$ it is natural to ask analogous questions for non-invertible maps and results here can also be viewed as a contribution along this line of research.  Let $C_{\lambda}(I)$ denote the family of Lebesgue measure preserving maps equipped with the metric of uniform convergence. For $C_{\lambda}(I)$, the genericity of shadowing and periodic shadowing was proven recently in \cite{BCOT}; therefore, the results obtained here can be viewed as strengthening of those results. However, we need to note that proving the genericity of s-limit shadowing turns out to be very delicate and, in particular, we cannot apply the main idea of the proof to the interval setting (see the explanation in Section~\ref{sec:Final}). 

Our result (with simplifications of the proof) holds in an even looser environment. By $C(\mathbb{S}^1)$ we denote the class of continuous interval maps endowed with the topology of uniform convergence. In this setting we obtain the following two new results.

\begin{thm}
	The s-limit shadowing property is generic in $C(\mathbb{S}^1)$.
\end{thm}

\begin{corollary}
	The limit shadowing property is generic in $C(\mathbb{S}^1)$.
\end{corollary}

Let us outline the structure of the paper. In Preliminaries we first review the definitions related to the shadowing property that we address in our context of Lebesgue measure preserving circle maps. Then we review the basic setting of $C_{\lambda}(\mathbb{S}^1)$ which we work with in the rest of the paper. We start Section~\ref{sec:main} by outlining the proof of main theorem. In Subsection~\ref{sec:SpecialFamilies} we restrict our attention to particular families of maps in $C_{\lambda}(\mathbb{S}^1)$ and we study their properties; we use these families and their properties later in the proof of s-limit shadowing. In Subsection~\ref{sec:PartitionsSpecialPerturbations} the proof of s-limit shadowing starts. We pose five conditions (C\ref{C1})-(C\ref{C5}) that our partitions and special perturbations need to satisfy. In the rest of this section we address how to get such partitions and perturbations from machinery developed in Subsection~\ref{sec:SpecialFamilies}. Subsection~\ref{sec:MainProof} gives the proof of Theorem~\ref{thm:main} using the assumptions given by conditions (C\ref{C1})-(C\ref{C5}) in Subsection~\ref{sec:PartitionsSpecialPerturbations}. We conclude the paper with Section~\ref{sec:Final} where we give a brief explanation why the proof of s-limit shadowing as presented in this paper can not work in the setting of Lebesgue measure preserving interval maps.

\section{Preliminaries}

Denote by $\N:=\{1,2,3,\ldots\}$ and $\N_0:=\N\cup\{0\}$. Let $\Ci:=\{z\in\C\colon~\vert z\vert=1\}$ be the \emph{unit circle}. For $x,y\in\mathbb{S}^1$ let $d(x,y)$ denote the minimal normalized arc-length distance on $\Ci$ between $x$ and $y$. 

\subsection{Shadowing property} First we give the definition of shadowing property and its related extensions that we use in this paper. 

\begin{definition}
For $\delta>0$ and a map $f\in C(\mathbb{S}^1)$ we say that a sequence of points $\{x_k\}_{k\in \N_0}\subset \mathbb{S}^1$ is a \emph{$\delta$-pseudo orbit}, if $d(f(x_k),x_{k+1})<\delta$ for all $k\in \N_0$. A $\delta$-pseudo orbit is called a \emph{periodic $\delta$-pseudo orbit} if there exists $N\in \N$ such that $x_{k+N}=x_k$ for all $k\in \N_0$.\\
A sequence $\{x_k\}_{k\in \N_0}\subset \mathbb{S}^1$
satisfying $\displaystyle \lim_{k\to\infty} d(f(x_{k}),x_{k+1}) = 0$ is called an \emph{asymptotic pseudo orbit}. \\
If a sequence $\{x_k\}_{k\in \N_0}\subset \mathbb{S}^1$ is a $\delta$-pseudo orbit and an asymptotic pseudo-orbit then we say that it is an \emph{asymptotic $\delta$-pseudo orbit}.
\end{definition}

\begin{definition}\label{def:shadowing}
	We say that a map $f\in C(\mathbb{S}^1)$ has the:
	\begin{itemize}
		\item \emph{shadowing property} if for every $\eps > 0$ there exists $\delta >0$ satisfying the following condition: given a $\delta$-pseudo orbit $\mathbf{y}=\{y_n\}_{n\in \N_0}$ we can find a corresponding point $x\in \mathbb{S}^1$ which $\eps$-traces $\mathbf{y}$, i.e.,
		$$d(f^n(x), y_n)<  \eps \text{ for every } n\in \N_0.$$
		\item
		\emph{periodic shadowing property} if for every $\eps>0$ there exists $\delta>0$ satisfying the following condition: given a periodic $\delta$-pseudo orbit $\mathbf{y}=\{y_n\}_{n\in\N_0}$ we can find a corresponding periodic point $x \in \mathbb{S}^1$, which $\eps$-traces $\mathbf{y}$.
		\item \emph{limit shadowing} if for every asymptotic pseudo orbit $\{x_n\}_{n\in \N_0}\subset \mathbb{S}^1$ there exists $p\in \mathbb{S}^1$ such that
		$$d(f^n(p),x_n)\to 0 \text{ as } n\to \infty.$$
		\item \emph{s-limit shadowing} if for every $\eps>0$ there exists $\delta>0$ so that
		\begin{enumerate}
			\item  for every $\delta$-pseudo orbit $\mathbf{y}=\{y_n\}_{n\in \N_0}$ we can find a corresponding point $x\in \mathbb{S}^1$ which $\eps$-traces $\mathbf{y}$,
			\item  for every asymptotic $\delta$-pseudo orbit $\mathbf{y}=\{y_n\}_{n\in \N_0}$ of $f$, there is $x\in \mathbb{S}^1$ which $\eps$-traces $\mathbf{y}$ and
			$$ \lim_{n\to \infty}d(y_n,f^n(x)) = 0.$$
		\end{enumerate}
	\end{itemize}
\end{definition}

\begin{remark}\label{rem:implies}
Note that s-limit shadowing implies both classical and limit shadowing.
\end{remark}

\subsection{Lebesgue measure preserving circle maps}
Consider a continuous map $f\colon~\mathbb S^1\to \mathbb S^1$ of
\emph{degree} $\degr(f)\in\Z$. Let $\tilde F\colon~\R\to\R$ be a \emph{lifting} of $f$, i.e., the continuous map
for which
\begin{equation}\label{e:10}\phi\circ \tilde F = f\circ\phi\text{ on }\R,
\end{equation}
where $\phi\colon~\R\to\Ci$ is defined by $\phi(x)= e^{2\pi ix}$. Then $\tilde F(x + 1) = \tilde F(x) + \degr(f)$ for each $x\in\R$.  If $F=\tilde F\vert [0,1)(\mo~ 1)$, we say that $F\colon~[0,1)\to [0,1)$ \emph{represents} $f$. Note that since two liftings of $f$ differ by a integer constant, $F$ does not depend on a concrete choice of a lifting of $f$.

In what follows the set of all liftings, resp.\  representatives of
{\it onto} circle maps will be denoted $\tilde \F(\R)$, resp.\  $\F([0,1))$.

\begin{remark}\label{r:1}One can easily see that a circle map $f$ is onto if and only if its representative $F=\tilde F\vert [0,1)(\mo~ 1)$ is onto.
\end{remark}

Let  $\tilde\lambda$ denote the \emph{normalized Lebesgue measure} on $\Ci$ and $\B$ the \emph{Borel sets} in $\Ci$. In this paper we will work with \emph{continuous maps from $\Ci$ into $\Ci$ preserving the measure $\tilde\lambda$}, which we denote
$$ C_{\tilde\lambda}(\Ci)=\{f\colon~\Ci\to \Ci\colon~\forall
A\in\B,~\tilde\lambda(A)=\tilde\lambda(f^{-1}(A))\}. $$
We consider the set $C_{\tilde\lambda}(\Ci)$ equipped with the \emph{uniform metric} $\rho$:
$$\rho (f,g) := \sup_{x \in \Ci} |f(x) - g(x)|.$$ We leave the standard proof of the following fact to the reader.

\begin{proposition}\label{p:1}$(C_{\tilde\lambda}(\Ci),\rho)$  is a complete metric space. \end{proposition}

The next lemma describes elements of $\F([0,1))$ representing maps from $C_{\tilde\lambda}(\Ci)$. We denote by $\lambda$ the  Lebesgue measure on $[0,1)$.

\begin{lemma}\label{l:3}Let $F\in \F([0,1))$ represent $f\colon~\Ci\to\Ci$. The following conditions are equivalent.
\begin{itemize}
\item[(i)] $f\in C_{\tilde\lambda}(\Ci)$.
\item[(ii)] $\forall A\subset [0,1)\text{ Borel },~\lambda(A)=\lambda(F^{-1}(A))$.
    \end{itemize}
 \end{lemma}

\begin{proof}Let us assume that $\tilde F$ is a lifting of $f$ and denote $\psi=\phi\vert [0,1)$. Then $\psi$ is a continuous bijection. From (\ref{e:10}) we get
\begin{equation}\label{e:11}\psi\circ F = f\circ\psi\text{ on }[0,1).
\end{equation}
Moreover, $A\subset [0,1)$ and $\tilde A:=\psi(A)\subset \Ci$ are simultaneously Borel and
\begin{equation}\label{e:12}\lambda(A)=\tilde\lambda(\tilde A).
\end{equation}
Assuming (i), using (\ref{e:11}) and (\ref{e:12}) we can write
$$\lambda(A)=\lambda(\psi^{-1}(\tilde A))=\lambda(\psi^{-1}(f^{-1}(\tilde A)))=\lambda(F^{-1}(\psi^{-1}(\tilde A)))=\lambda(F^{-1}(A)).
$$
This shows that the statements (i) implies (ii). If (ii) is true we can write
\begin{align*}&\tilde\lambda(\tilde A)=\lambda(A)=\lambda(F^{-1}(A))=\lambda(F^{-1}(\psi^{-1}(\tilde A)))=\lambda(\psi^{-1}(f^{-1}(\tilde A)))=\tilde\lambda(f^{-1}(\tilde A)),
\end{align*}
so (ii) implies (i).
\end{proof}

We say that a map from $\F([0,1))$ is \emph{piecewise affine} if it has finitely many affine pieces of monotonicity. We will say that $\tilde F\in\tilde \F(\R)$ is \emph{piecewise affine} if its corresponding representative $\tilde F\vert[0,1))$ is piecewise affine.
In general, maps from $\F([0,1))$ are not continuous but they can be piecewise monotone and smooth or even piecewise affine.  For these cases the following lemma states a useful criterium about when an element $F$ of $\F([0,1))$ represents $f\in C_{\tilde\lambda}(\Ci)$.

\begin{lemma}\label{l:6}
Let $F\in \F([0,1))$ be a piecewise affine representative with nonzero slopes and such that its derivative does not exist at a finite set $E$. Then the properties (i) and (ii) from Lemma \ref{l:3} are equivalent to the property
\begin{equation}\label{e:3}
   \forall~y\in [0,1)\setminus F(E)\colon~\sum_{x\in F^{-1}(y)}\frac{1}{\vert F'(x)\vert}=1.
 \end{equation}
 \end{lemma}\begin{proof}By the hypothesis the set $F(E)$ is finite and for each $y\in (0,1)\setminus F(E)$ we can write for $J(y,\eps)=[y-\eps,y+\eps]$
\begin{equation*}\label{e:13}\lim_{\eps\to 0_+}\sum_{K\in\comp(F^{-1}(J(y,\eps)))}\frac{\lambda(K)}{\lambda(J(y,\eps))}=\sum_{x\in F^{-1}(y)}\frac{1}{\vert F'(x)\vert},\end{equation*}
thus Lemma \ref{l:3}(ii) implies (\ref{e:3}). 

For the other direction, assuming that property (ii) is not true, one can find some closed interval $J_0\subset [0,1)\setminus F(E)$ and $\delta>0$ for which
\begin{equation}\label{e:14}\frac{\lambda(J_0)}{\lambda(F^{-1}(J_0))}\in (0,1-\delta)\cup (1+\delta,\infty);\end{equation} then proceeding inductively 
we can detect a nested sequence $J_0\supset J_1\supset \cdots $ of closed intervals fulfilling (\ref{e:14}) and $\bigcap_{i=0}^{\infty} J_i=\{y\}$ 
with $y\in (0,1)\setminus F(E)$. By (\ref{e:14}), equation (\ref{e:3}) fails in such $y$. It shows that (\ref{e:3}) implies Lemma \ref{l:3}(ii).
 \end{proof}
 
\section{The proof}\label{sec:main}

\subsection{Outline of the proof}\label{sec:outline}
The proof of our main result, Theorem \ref{thm:main}, relies on four rather technical steps. The first step is treated in Lemmas \ref{l:7}, \ref{l:31} and  \ref{l:5} and consists of the construction of a special dense subset $C_{\tilde\lambda,0}(\Ci)$ of $C_{\tilde\lambda}(\Ci)$. Let $\Q_{\pi} := \Q + \pi$. The maps in $C_{\tilde\lambda,0}(\Ci)$ are piecewise affine 
and every map $g$ from $C_{\tilde\lambda,0}(\Ci)$ fulfills the key property 
\begin{equation}\label{e:uffo}g(\phi(\Q_{\pi}))\subset \phi(\Q).\end{equation} 
In particular, the maps in $C_{\tilde\lambda,0}(\Ci)$ have all points of discontinuity of  derivatives in $\phi(\Q_{\pi})$ so the equation (\ref{e:uffo}) applies.       
In the second step in Lemmas \ref{l:1} and \ref{l:2} we twice perturb  maps from $C_{\tilde\lambda,0}(\Ci)$  to obtain maps satisfying the list of conditions (C\ref{C1})-(C\ref{C5}) from Subsection~\ref{sec:PartitionsSpecialPerturbations}. Applying in Subsection \ref{sec:MainProof} both perturbations and also the result of Lemma \ref{l:4} on a sequence $\{g_m\}_{m\ge 1}$ dense in $C_{\tilde\lambda,0}(\Ci)$, we arrive to new sequences $\{\theta_m\}_{m\ge 1}$ of maps from $C_{\tilde\lambda}(\Ci)$, their neighborhoods $\{U_m\}_{m\ge 1}$ and also carefully constructed partitions $\{\PQ_m\}_{m\ge 1}$. In particular, using (\ref{e:uffo}) we can ensure that for some pairs $m<n$, $\PQ_n$ is a refinement of $\PQ_m$. The final step consists of the proof that all maps in 
$$A=\bigcap_{n\ge 1}\bigcup_{m\ge n}U_m$$
have the s-limit shadowing property.

\subsection{Particular families of Lebesgue measure preserving circle maps and its representatives}\label{sec:SpecialFamilies}
In this subsection we will define particular families of Lebesgue measure preserving circle maps that we will apply later for the construction of partitions needed for the proof of genericity of s-limit shadowing property in our context.

For a piecewise affine map $\tilde F\in\tilde \F(\R)$ (i.e. $\tilde F\vert[0,1))$ is piecewise affine) we denote by $T(\tilde F)$, resp.\ $D(\tilde F)$ the \emph{turning points}, resp.\ the \emph{set of points of discontinuity of derivative} of $\tilde F$. We also put $T_{[0,1]}(\tilde F)=T(\tilde F)\cap [0,1]$ and denote
$$\Q_{\pi}=\{r+\pi\colon~r\in\Q\}.
$$

Let us recall our convention that is stated before Remark \ref{r:1}; the set $\tilde\F(\R)$ consists of liftings of {\it onto} circle maps. Let $\tilde\F_0(\R)\subset\tilde\F(\R)$ be defined as
\begin{equation}\label{e:41}
\tilde F\in\tilde\F_0(\R)~\equiv~
\begin{cases}
\text{(i)}~\tilde F\vert[0,1)\text{ is piecewise affine with nonzero slopes},\\
\text{(ii) for}~S=\{0\}\cup F(T_{[0,1]}(\tilde F)),~~D(\tilde F)\cap [0,1]= F^{-1}(S)\subset \Q_{\pi}.

\end{cases}
\end{equation}

Since the set $\Q_{\pi}$ is dense in $\R$, we have the following lemma.
\begin{lemma}\label{l:7}The set $\tilde \F_0(\R)$ is dense in $\tilde \F(\R)$.
\end{lemma}
\begin{proof}Fix $\hat F\in\tilde\F(\R)$ and $\eps>0$. Clearly there exists a piecewise affine map $\tilde F\in\tilde\F(\R)$  with nonzero slopes such that \begin{itemize}
\item $\sup_{x\in [0,1]}\vert \hat F(x)-\tilde F(x)\vert<\eps$,
\item $\hat F(1)-\hat F(0)=\tilde F(1)-\tilde F(0)$,
\item $T(\tilde F)\subset\Q_{\pi}$,
\item for $F=\tilde F\vert [0,1)(\mo~ 1)$, if $x\in (0,1)$ satisfies $F(x)\in S$ then $x\in D(\tilde F)\cap\Q_{\pi}$, so $(D(\tilde F)\cap [0,1])\supset F^{-1}(S)$.
\end{itemize}

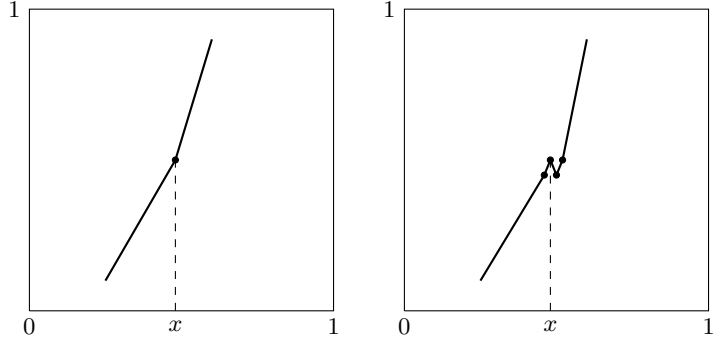
\begin{figure}[!ht]
	\centering
	\begin{tikzpicture}[scale=4]
	\draw (0,0)--(0,1)--(1,1)--(1,0)--(0,0);
	\draw[dashed] (0.48,0)--(0.48,1/2);
	\draw[thick] (1/4,1/10)--(0.48,1/2)--(6/10,9/10);
	\draw[fill] (0.48,1/2) circle (0.01);
	\node at (0.48,-0.05) {\small $x$};
	\node at (-0.05,1) {\small $1$};
	\node at (0,-0.05) {\small $0$};
	\node at (1,-0.05) {\small  $1$};
	\end{tikzpicture}
	\hspace{0,1cm}
	\begin{tikzpicture}[scale=4]
	\draw (0,0)--(0,1)--(1,1)--(1,0)--(0,0);
	\draw[dashed] (0.48,0)--(0.48,1/2);
	\draw[thick] (1/4,1/10)--(0.46,9/20)--(0.48,1/2)--(0.5,9/20)--(0.52,1/2)--(6/10,9/10);
	\draw[fill] (0.48,1/2) circle (0.01);
	\draw[fill] (0.46,9/20) circle (0.01);
	\draw[fill] (0.50,9/20) circle (0.01);
	\draw[fill] (0.52,1/2) circle (0.01);
	\node at (0.48,-0.05) {\small $x$};
	\node at (-0.05,1) {\small $1$};
	\node at (0,-0.05) {\small $0$};
	\node at (1,-0.05) {\small  $1$};
	\end{tikzpicture}
	\caption{Adjustments from the proof of Lemma~\ref{l:7}.}\label{Fig1}
\end{figure}

Notice that for a piecewise affine $\tilde F$ the set $(D(\tilde F)\cap [0,1])\setminus F^{-1}(S)$ has to be either empty or finite. For any $x\in (D(\tilde F)\cap [0,1])\setminus F^{-1}(S)$ we can proceed in two steps. First, 
we modify the graph of $F$ on a small neighbourhood of $x$ as it is shown in Figure~\ref{Fig1}. Second, denoting the new maps $F$, $\tilde F$, arrange both new turning points and also all preimages of their images to be from $D(\tilde F)\cap \Q_{\pi}$ and therefore reduce the number
$$\#(D((\tilde F)\cap [0,1])\setminus F^{-1}(S)).
$$
Repeating the described modification finitely many times, we fulfill (\ref{e:41})(ii), i.e., $\tilde F\in \tilde \F_0(\R)$.
\end{proof}

Consider a lifting $\tilde F\in\tilde F_0(\R)$ introduced by (\ref{e:41}) and the corresponding representative $F=\tilde F\vert [0,1)(\mo~1)$. Define the \emph{outer homeomorphism} $h:[0,1]\to [0,1]$~by
\begin{equation}\label{e:2}
h(0)=0\text{ and }h(x)=\lambda(F^{-1}((0,x))),~x\in (0,1].
\end{equation}

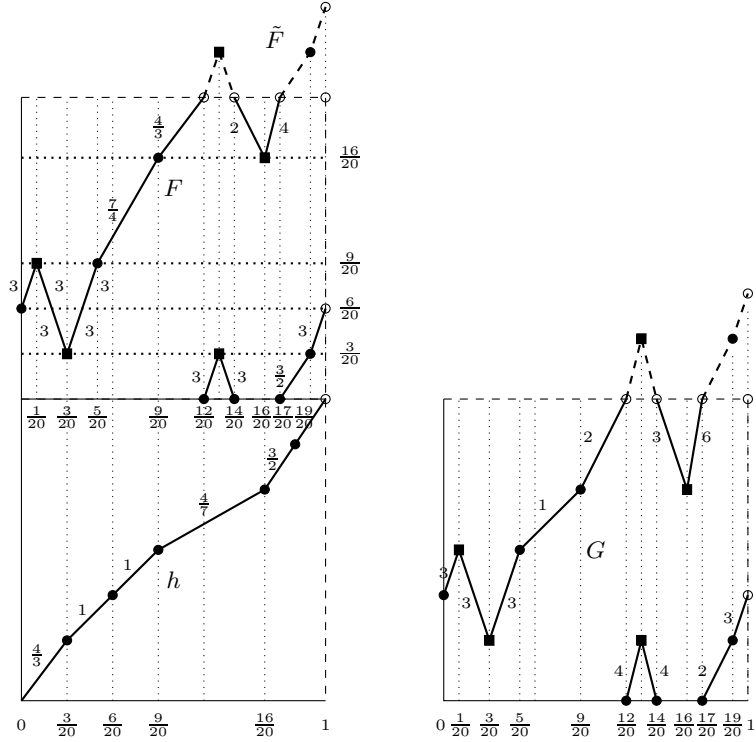
\begin{figure}[!ht]
	\centering
	\begin{tikzpicture}[scale=4]
	\draw (0,1)--(0,0)--(1,0);
	\draw[dashed] (0,1)--(1,1)--(1,0);
	\draw[dotted] (1/20,1)--(1/20,2);
	\draw[dotted] (3/20,0)--(3/20,2);
	\draw[dotted] (5/20,1)--(5/20,2);
	\draw[dotted] (6/20,0)--(6/20,2);
	\draw[dotted] (9/20,0)--(9/20,2);
	\draw[dotted] (12/20,0)--(12/20,2);
	\draw[dotted] (13/20,1)--(13/20,2+3/20);
	\draw[dotted] (14/20,1)--(14/20,2);
	\draw[dotted] (16/20,0)--(16/20,2);
	\draw[dotted] (17/20,1)--(17/20,2);
	\draw[dotted] (19/20,1)--(19/20,2+3/20);
	\draw[dotted] (1,1)--(1,2+6/20);
	\draw[dotted,thick] (0,1+3/20)--(1,1+3/20);
	\draw[dotted,thick] (0,1+6/20)--(1,1+6/20);
	\draw[dotted,thick] (0,1+9/20)--(1,1+9/20);
	\draw[dotted,thick] (0,1+16/20)--(1,1+16/20);
	\draw (1,1)--(0,1)--(0,2);
	\draw [dashed] (0,2)--(1,2)--(1,1);
	\node at (1/2,1.7){\small $F$};
	\node at (5/6,2.2){\small $\tilde F$};
	\draw[thick] (0,0)--(3/20,4/20)--(9/20,10/20)--(16/20,7/10)--(1,1);
	\draw[thick] (0,6/20+1)--(1/20,9/20+1)--(3/20,3/20+1)--(5/20,9/20+1)--(9/20,16/20+1)--(12/20,1+1);
	\draw[dashed,thick] (12/20,1+1)--(13/20,23/20+1)--(14/20,1+1);
	\draw[dashed,thick] (17/20,1+1)--(19/20,23/20+1)--(1,26/20+1);
	\draw[thick] (14/20,1+1)--(16/20,16/20+1)--(17/20,1+1);
	\draw[thick] (12/20,0+1)--(13/20,3/20+1)--(14/20,0+1);
	\draw[thick] (17/20,0+1)--(19/20,3/20+1)--(1,6/20+1);
	\node at (1/20,3/20) {\tiny $\frac{4}{3}$};
	\node at (4/20,6/20) {\tiny $1$};
	\node at (7/20,9/20) {\tiny $1$};
	\node at (12/20,13/20) {\tiny $\frac{4}{7}$};
	\node at (16.5/20,16/20) {\tiny $\frac{3}{2}$};
	\node at (6/20,1+12.5/20) {\tiny $\frac{7}{4}$};
	\node at (17.3/20,1+18/20) {\tiny $4$};
	\node at (14/20,1+18/20) {\tiny $2$};
	\node at (9/20,1+18/20) {\tiny $\frac{4}{3}$};
	\node at (-0.5/20,1+7.5/20) {\tiny $3$};
	\node at (2.5/20,1+7.5/20) {\tiny $3$};
	\node at (5.5/20,1+7.5/20) {\tiny $3$};
	\node at (1.5/20,1+4.5/20) {\tiny $3$};
	\node at (4.5/20,1+4.5/20) {\tiny $3$};
	\node at (18.5/20,1+4.5/20) {\tiny $3$};
	\node at (11.5/20,1+1.5/20) {\tiny $3$};
	\node at (14.5/20,1+1.5/20) {\tiny $3$};
	\node at (17/20,1+1.5/20) {\tiny $\frac{3}{2}$};
	\node at (0,-0.08) {\tiny $0$};
	\node at (1,-0.08) {\tiny $1$};
	\node at (3/20,-0.08) {\tiny $\frac{3}{20}$};
	\node at (6/20,-0.08) {\tiny $\frac{6}{20}$};
	\node at (9/20,-0.08) {\tiny $\frac{9}{20}$};
	\node at (16/20,-0.08) {\tiny $\frac{16}{20}$};
	\node at (1/20,1-0.06) {\tiny $\frac{1}{20}$};
	\node at (3/20,1-0.06) {\tiny $\frac{3}{20}$};
	\node at (5/20,1-0.06) {\tiny $\frac{5}{20}$};
	\node at (9/20,1-0.06) {\tiny $\frac{9}{20}$};
	\node at (12/20,1-0.06) {\tiny $\frac{12}{20}$};
	\node at (14/20,1-0.06) {\tiny $\frac{14}{20}$};
	\node at (15.8/20,1-0.06) {\tiny $\frac{16}{20}$};
	\node at (17.2/20,1-0.06) {\tiny $\frac{17}{20}$};
	\node at (18.6/20,1-0.06) {\tiny $\frac{19}{20}$};
	\node at (1+0.08,1+3/20) {\tiny $\frac{3}{20}$};
	\node at (1+0.08,1+6/20) {\tiny $\frac{6}{20}$};
	\node at (1+0.08,1+9/20) {\tiny $\frac{9}{20}$};
	\node at (1+0.08,1+16/20) {\tiny $\frac{16}{20}$};
	\node at (1/2,2/5){\small $h$};
	\draw[fill] (3/20,1/5) circle (0.015);
	\draw[fill] (6/20,7/20) circle (0.015);
	\draw[fill] (9/20,10/20) circle (0.015);
	\draw[fill] (16/20,14/20) circle (0.015);
	\draw[fill] (18/20,17/20) circle (0.015);
	\draw (1,1) circle (0.015);
	\draw[fill] (12/20,1) circle (0.015);
	\draw[fill] (14/20,1) circle (0.015);
	\draw[fill] (17/20,1) circle (0.015);
	\draw[fill] (19/20,1+3/20) circle (0.015);
	\draw (1,1+6/20) circle (0.015);
	\draw[fill] (0,6/20+1) circle (0.015);
	\draw[fill] (5/20,9/20+1) circle (0.015);
	\draw[fill] (9/20,16/20+1) circle (0.015);
	\draw (12/20,1+1) circle (0.015);
	\draw (14/20,1+1) circle (0.015);
	\draw (17/20,1+1) circle (0.015);
	\draw (1,1+1) circle (0.015);
	\draw[fill] (19/20,23/20+1) circle (0.015);
	\draw (1,2+6/20) circle (0.015);
	\draw [fill] (0.035,1+9.325/20) rectangle (0.065,1+8.675/20);
	\draw [fill] (0.035+2/20,1+9.325/20-6/20) rectangle (0.065+2/20,1+8.675/20-6/20);
	\draw [fill] (0.035+2/20+1/2,1+9.325/20-6/20) rectangle (0.065+2/20+1/2,1+8.675/20-6/20);
	\draw [fill] (0.035+2/20+1/2,1+9.325/20-6/20+1) rectangle (0.065+2/20+1/2,1+8.675/20-6/20+1);
	\draw [fill] (0.035+3/4,1+9.325/20+7/20) rectangle (0.065+3/4,1+8.675/20+7/20);
	\end{tikzpicture}
	\hspace{0,5cm}
	\begin{tikzpicture}[scale=4]
	\draw (0,1)--(0,0)--(1,0);
	\draw[dashed] (0,1)--(1,1)--(1,0);
	\draw[thick] (0,6/20+1/20)--(1/20,9/20+1/20)--(3/20,3/20+1/20)--(5/20,9/20+1/20)--(9/20,13/20+1/20)--(12/20,1);
	\draw[dashed,thick] (12/20,1)--(13/20,23/20+1/20)--(14/20,1);
	\draw[dashed,thick] (17/20,1)--(19/20,23/20+1/20)--(1,26/20+1/20);
	\draw[thick] (14/20,1)--(16/20,13/20+1/20)--(17/20,1);
	\draw[thick] (12/20,0)--(13/20,3/20+1/20)--(14/20,0);
	\draw[thick] (17/20,0)--(19/20,3/20+1/20)--(1,6/20+1/20);
	\draw[fill] (12/20,0) circle (0.015);
	\draw[fill] (14/20,0) circle (0.015);
	\draw[fill] (17/20,0) circle (0.015);
	\draw[dotted] (1/20,0)--(1/20,1);
	\draw[dotted] (3/20,0)--(3/20,1);
	\draw[dotted] (5/20,0)--(5/20,1);
	\draw[dotted] (6/20,0)--(6/20,1);
	\draw[dotted] (9/20,0)--(9/20,1);
	\draw[dotted] (12/20,0)--(12/20,1);
	\draw[dotted] (13/20,0)--(13/20,1+3/20);
	\draw[dotted] (14/20,0)--(14/20,1);
	\draw[dotted] (16/20,0)--(16/20,1);
	\draw[dotted] (17/20,0)--(17/20,1);
	\draw[dotted] (19/20,0)--(19/20,1+3/20);
	\draw[dotted] (1,0)--(1,1+6/20);
	\node at (-0.01,-0.08) {\tiny $0$};
	\node at (1.01,-0.08) {\tiny $1$};
	\node at (1.1/20,-0.08) {\tiny $\frac{1}{20}$};
	\node at (3/20,-0.08) {\tiny $\frac{3}{20}$};
	\node at (5/20,-0.08) {\tiny $\frac{5}{20}$};
	\node at (9/20,-0.08) {\tiny $\frac{9}{20}$};
	\node at (12/20,-0.08) {\tiny $\frac{12}{20}$};
	\node at (14/20,-0.08) {\tiny $\frac{14}{20}$};
	\node at (15.8/20,-0.08) {\tiny $\frac{16}{20}$};
	\node at (17.3/20,-0.08) {\tiny $\frac{17}{20}$};
	\node at (19/20,-0.08) {\tiny $\frac{19}{20}$};
	\draw[fill] (19/20,3/20+1/20) circle (0.015);
	\draw (1,6/20+1/20) circle (0.015);
	\draw[fill] (0,6/20+1/20) circle (0.015);
	\draw[fill] (5/20,9/20+1/20) circle (0.015);
	\draw[fill] (9/20,13/20+1/20) circle (0.015);
	\draw (12/20,1) circle (0.015);
	\draw (14/20,1) circle (0.015);
	\draw (17/20,1) circle (0.015);
	\draw (1,1) circle (0.015);
	\draw[fill] (19/20,23/20+1/20) circle (0.015);
	\draw (1,1+6/20+1/20) circle (0.015);
	\draw [fill] (0.035,9.325/20+1/20) rectangle (0.065,8.675/20+1/20);
	\draw [fill] (0.035+2/20,9.325/20-6/20+1/20) rectangle (0.065+2/20,8.675/20-6/20+1/20);
	\draw [fill] (0.035+2/20+1/2,9.325/20-6/20+1/20) rectangle (0.065+2/20+1/2,8.675/20-6/20+1/20);
	\draw [fill] (0.035+2/20+1/2,9.325/20-6/20+1/20+1) rectangle (0.065+2/20+1/2,8.675/20-6/20+1/20+1);
	\draw [fill] (0.035+3/4,9.325/20+6/20-1/20) rectangle (0.065+3/4,8.675/20+6/20-1/20);
	\node at (17.3/20,17.5/20) {\tiny $6$};
	\node at (14/20,17.5/20) {\tiny $3$};
	\node at (9.5/20,17.5/20) {\tiny $2$};
	\node at (0/20,8.5/20) {\tiny $3$};
	\node at (6.5/20,13/20) {\tiny $1$};
	\node at (1.5/20,6.5/20) {\tiny $3$};
	\node at (4.5/20,6.5/20) {\tiny $3$};
	\node at (11.5/20,2/20) {\tiny $4$};
	\node at (14.5/20,2/20) {\tiny $4$};
	\node at (17/20,2/20) {\tiny $2$};
	\node at (18.7/20,5.5/20) {\tiny $3$};
	\node at (1/2,1/2){\small $G$};
	\end{tikzpicture}
	\caption{On these pictures the numbers along the graph lines represent slopes of respective affinity pieces. On the left upper part of the figure, there is a picture of a graph of a lifting of a non-Lebesgue preserving circle map $\tilde F$ restricted on $[0,1)$ (taking into account the dashed lines) and also of $F$, its corresponding representative (without the dashed lines). On the left lower picture there is a depiction of the corresponding outer homeomorphism $h$. The picture on the right represents a Lebesgue measure preserving map $G$, however the lifting of this map is not from the set $\tilde{\mathcal{F}}_0$ since the maps $\tilde F$ and $F$ do not have their turning points (black squares) and also preimages of images of turning points that are not turning points (black discs) in $\Q_{\pi}$.}\label{fig:original}
\end{figure}

Clearly, by (\ref{e:41}) and Remark \ref{r:1} the map $F$ is surjective with nonzero slopes, $h$ is an increasing continuous piecewise affine function satisfying $h(0)=0$ and $h(1)=1$. In particular, $h$ is a homeomorphism of $[0,1]$. The set of all liftings of maps from $C_{\tilde\lambda}(\Ci)$ will be denoted by $\tilde\F_{\lambda}(\R)$.

 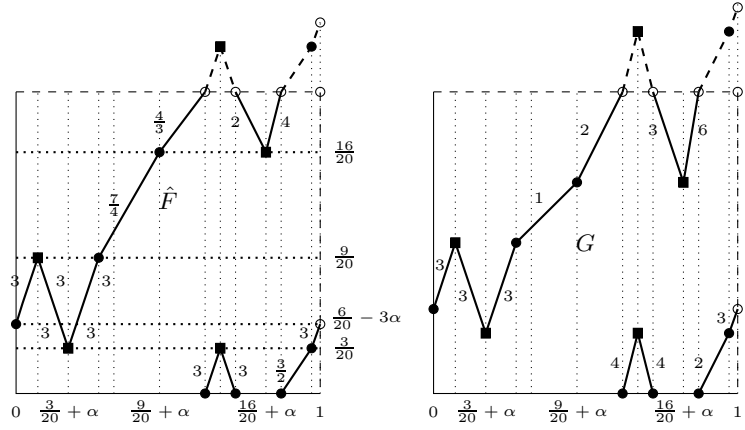
\begin{figure}[!ht]
 	\centering
 	\begin{tikzpicture}[scale=4]
	\draw (0,1)--(0,0)--(1,0);
	\draw[dashed] (0,1)--(1,1)--(1,0);
 	\draw[dotted] (1/20+\alp,0)--(1/20+\alp,1);
 	\draw[dotted] (3/20+\alp,0)--(3/20+\alp,1);
 	\draw[dotted] (5/20+\alp,0)--(5/20+\alp,1);
 	\draw[dotted] (6/20+\alp,0)--(6/20+\alp,1);
 	\draw[dotted] (9/20+\alp,0)--(9/20+\alp,1);
 	\draw[dotted] (12/20+\alp,0)--(12/20+\alp,1);
 	\draw[dotted] (13/20+\alp,0)--(13/20+\alp,1+3/20);
 	\draw[dotted] (14/20+\alp,0)--(14/20+\alp,1);
 	\draw[dotted] (16/20+\alp,0)--(16/20+\alp,1);
 	\draw[dotted] (17/20+\alp,0)--(17/20+\alp,1);
 	\draw[dotted] (19/20+\alp,0)--(19/20+\alp,1+3/20);
 	\draw[dotted] (1,0)--(1,1+0.23);
 	\draw[dotted,thick] (0,3/20)--(1,3/20);
 	\draw[dotted,thick] (0,0.23)--(1,0.23);
 	\draw[dotted,thick] (0,9/20)--(1,9/20);
 	\draw[dotted,thick] (0,16/20)--(1,16/20);	
 	\draw[thick] (0,0.23)--(0+\alp,6/20)--(1/20+\alp,9/20)--(3/20+\alp,3/20)--(5/20+\alp,9/20)--(9/20+\alp,16/20)--(12/20+\alp,1);
 	\draw[dashed,thick] (12/20+\alp,1)--(13/20+\alp,23/20)--(14/20+\alp,1);
 	\draw[dashed,thick] (17/20+\alp,1)--(19/20+\alp,23/20)--(1,1.23);
 	\draw[thick] (14/20+\alp,1)--(16/20+\alp,16/20)--(17/20+\alp,1);
 	\draw[thick] (12/20+\alp,0)--(13/20+\alp,3/20)--(14/20+\alp,0);
 	\draw[thick] (17/20+\alp,0)--(19/20+\alp,3/20)--(1,0.23);
 	\node at (6/20+\alp,12.5/20) {\tiny $\frac{7}{4}$};
 	\node at (17.3/20+\alp,18/20) {\tiny $4$};
 	\node at (14/20+\alp,18/20) {\tiny $2$};
 	\node at (9/20+\alp,18/20) {\tiny $\frac{4}{3}$};
 	\node at (-0.5/20+\alp,7.5/20) {\tiny $3$};
 	\node at (2.5/20+\alp,7.5/20) {\tiny $3$};
 	\node at (5.5/20+\alp,7.5/20) {\tiny $3$};
 	\node at (1.5/20+\alp,4/20) {\tiny $3$};
 	\node at (4.5/20+\alp,4/20) {\tiny $3$};
 	\node at (18.5/20+\alp,4/20) {\tiny $3$};
 	\node at (11.5/20+\alp,1.5/20) {\tiny $3$};
 	\node at (14.5/20+\alp,1.5/20) {\tiny $3$};
 	\node at (17/20+\alp,1.5/20) {\tiny $\frac{3}{2}$};
 	\node at (0,-0.06) {\tiny $0$};
 	\node at (1,-0.06) {\tiny $1$};
 	\node at (3/20+\alp,-0.06) {\tiny $\frac{3}{20}+\alpha$};
 	\node at (9/20+\alp,-0.06) {\tiny $\frac{9}{20}+\alpha$};
 	\node at (16/20+\alp,-0.06) {\tiny $\frac{16}{20}+\alpha$};
 	\node at (1+0.08,3/20) {\tiny $\frac{3}{20}$};
 	\node at (1+0.15,5/20) {\tiny $\frac{6}{20}-3\alpha$};
 	\node at (1+0.08,9/20) {\tiny $\frac{9}{20}$};
 	\node at (1+0.08,16/20) {\tiny $\frac{16}{20}$};
 	\node at (1/2,3/4-0.1) {\small $\hat F$};
 	\draw[fill] (19/20+\alp,3/20) circle (0.015);
 	\draw (1,0.23) circle (0.015);
 	\draw[fill] (0,0.23) circle (0.015);
 	\draw[fill] (5/20+\alp,9/20) circle (0.015);
 	\draw[fill] (9/20+\alp,16/20) circle (0.015);
 	\draw (12/20+\alp,1) circle (0.015);
 	\draw (14/20+\alp,1) circle (0.015);
 	\draw (17/20+\alp,1) circle (0.015);
 	\draw[fill] (12/20+\alp,0) circle (0.015);
 	\draw[fill] (14/20+\alp,0) circle (0.015);
 	\draw[fill] (17/20+\alp,0) circle (0.015);
 	\draw (1,1) circle (0.015);
 	\draw[fill] (19/20+\alp,23/20) circle (0.015);
 	\draw (1,1.23) circle (0.015);
 	\draw [fill] (0.035+\alp,9.325/20) rectangle (0.065+\alp,8.675/20);
 	\draw [fill] (0.035+2/20+\alp,9.325/20-6/20) rectangle (0.065+2/20+\alp,8.675/20-6/20);
 	\draw [fill] (0.035+2/20+1/2+\alp,9.325/20-6/20) rectangle (0.065+2/20+1/2+\alp,8.675/20-6/20);
 	\draw [fill] (0.035+2/20+1/2+\alp,9.325/20-6/20+1) rectangle (0.065+2/20+1/2+\alp,8.675/20-6/20+1);
 	\draw [fill] (0.035+3/4+\alp,9.325/20+7/20) rectangle (0.065+3/4+\alp,8.675/20+7/20);
 	\end{tikzpicture}
 	\begin{tikzpicture}[scale=4]
 	\draw (0,1)--(0,0)--(1,0);
 	\draw[dashed] (0,1)--(1,1)--(1,0);
 	\draw[thick] (0,0.23+1/20)--(1/20+\alp,9/20+1/20)--(3/20+\alp,3/20+1/20)--(5/20+\alp,9/20+1/20)--(9/20+\alp,13/20+1/20)--(12/20+\alp,1);
 	\draw[dashed,thick] (12/20+\alp,1)--(13/20+\alp,23/20+1/20)--(14/20+\alp,1);
 	\draw[dashed,thick] (17/20+\alp,1)--(19/20+\alp,23/20+1/20)--(1,1.23+1/20);
 	\draw[thick] (14/20+\alp,1)--(16/20+\alp,13/20+1/20)--(17/20+\alp,1);
 	\draw[thick] (12/20+\alp,0)--(13/20+\alp,3/20+1/20)--(14/20+\alp,0);
 	\draw[thick] (17/20+\alp,0)--(19/20+\alp,3/20+1/20)--(1,0.23+1/20);
 	\draw[fill] (12/20+\alp,0) circle (0.015);
 	\draw[fill] (14/20+\alp,0) circle (0.015);
 	\draw[fill] (17/20+\alp,0) circle (0.015);
 	\draw[dotted] (1/20+\alp,0)--(1/20+\alp,1);
 	\draw[dotted] (3/20+\alp,0)--(3/20+\alp,1);
 	\draw[dotted] (5/20+\alp,0)--(5/20+\alp,1);
 	\draw[dotted] (6/20+\alp,0)--(6/20+\alp,1);
 	\draw[dotted] (9/20+\alp,0)--(9/20+\alp,1);
 	\draw[dotted] (12/20+\alp,0)--(12/20+\alp,1);
 	\draw[dotted] (13/20+\alp,0)--(13/20+\alp,1+3/20);
 	\draw[dotted] (14/20+\alp,0)--(14/20+\alp,1);
 	\draw[dotted] (16/20+\alp,0)--(16/20+\alp,1);
 	\draw[dotted] (17/20+\alp,0)--(17/20+\alp,1);
 	\draw[dotted] (19/20+\alp,0)--(19/20+\alp,1+3/20);
 	\draw[dotted] (1,0)--(1,1+6/20);
 	\node at (0,-0.06) {\tiny $0$};
 	\node at (1,-0.06) {\tiny $1$};
 	\node at (3/20+\alp,-0.06) {\tiny $\frac{3}{20}+\alpha$};
 	\node at (9/20+\alp,-0.06) {\tiny $\frac{9}{20}+\alpha$};
 	\node at (16/20+\alp,-0.06) {\tiny $\frac{16}{20}+\alpha$};
 	\draw[fill] (19/20+\alp,3/20+1/20) circle (0.015);
 	\draw (1,0.23+1/20) circle (0.015);
 	\draw[fill] (0,0.23+1/20) circle (0.015);
 	\draw[fill] (5/20+\alp,9/20+1/20) circle (0.015);
 	\draw[fill] (9/20+\alp,13/20+1/20) circle (0.015);
 	\draw (12/20+\alp,1) circle (0.015);
 	\draw (14/20+\alp,1) circle (0.015);
 	\draw (17/20+\alp,1) circle (0.015);
 	\draw (1,1) circle (0.015);
 	\draw[fill] (19/20+\alp,23/20+1/20) circle (0.015);
 	\draw (1,1+0.23+1/20) circle (0.015);
 	\draw [fill] (0.035+\alp,9.325/20+1/20) rectangle (0.065+\alp,8.675/20+1/20);
 	\draw [fill] (0.035+2/20+\alp,9.325/20-6/20+1/20) rectangle (0.065+2/20+\alp,8.675/20-6/20+1/20);
 	\draw [fill] (0.035+2/20+1/2+\alp,9.325/20-6/20+1/20) rectangle (0.065+2/20+1/2+\alp,8.675/20-6/20+1/20);
 	\draw [fill] (0.035+2/20+1/2+\alp,9.325/20-6/20+1/20+1) rectangle (0.065+2/20+1/2+\alp,8.675/20-6/20+1/20+1);
 	\draw [fill] (0.035+3/4+\alp,9.325/20+6/20-1/20) rectangle (0.065+3/4+\alp,8.675/20+6/20-1/20);
 	\node at (17.3/20+\alp,17.5/20) {\tiny $6$};
 	\node at (14/20+\alp,17.5/20) {\tiny $3$};
 	\node at (9.5/20+\alp,17.5/20) {\tiny $2$};
 	\node at (0/20+\alp,8.5/20) {\tiny $3$};
 	\node at (6.5/20+\alp,13/20) {\tiny $1$};
 	\node at (1.5/20+\alp,6.5/20) {\tiny $3$};
 	\node at (4.5/20+\alp,6.5/20) {\tiny $3$};
 	\node at (11.5/20+\alp,2/20) {\tiny $4$};
 	\node at (14.5/20+\alp,2/20) {\tiny $4$};
 	\node at (17/20+\alp,2/20) {\tiny $2$};
 	\node at (18.5/20+\alp,5/20) {\tiny $3$};
 	\node at (1/2,1/2){\small $G$};
 	\end{tikzpicture}
 	\caption{Let $r\in \mathbb{Q}$ and let $\alpha=\pi-r>0$ be a small irrational number. On the left picture, graph of function $\hat F$ represents a shift (i.e. rotation on the circle for the original circle map) of the representative $F$ from Figure~\ref{fig:original} for $\alpha$ to the right (and its lift, similarly as in Figure~\ref{fig:original}).  Due to the choice of $\alpha$, the lifting $\tilde F(x+\alpha)$ will already be from $\tilde\F_0(\R)$. Note that the outer homeomorphism for $\hat F$ stays the same as the one in Figure~\ref{fig:original}.}\label{fig:shifted}
 \end{figure}

\begin{lemma}\label{l:31}Let $\tilde F \in \tilde \F_0(\R)$ be lifting of $f\in C(\Ci)$, $F$ its corresponding representative and $h$ defined as in (\ref{e:2}). For the map $G=h\circ F$ the following is true.
\begin{itemize}
\item[(i)] $\forall A\subset [0,1)\text{ Borel },~\lambda(A)=\lambda(G^{-1}(A))$.
\item[(ii)] $G^{-1}(0)=F^{-1}(0)$.
\item[(iii)] The function $\hat G\colon~[0,1)\to \R$ defined by $\hat G(x)=G(x)+\tilde F(x)-F(x)$ is piecewise affine, continuous and $\lim_{x\to 1_-}\hat G(x)=\hat G(0)+\degr(f)$.
\item[(iv)] The function $\tilde G\colon~\R\to \R$ defined as the extension of $\hat G$ satisfying $$\tilde G(x+1)=\tilde G(x)+\degr(f)$$ belongs to $\tilde\F_{\lambda}(\R)$, so $\tilde G$ is a lifting of some $g\in C_{\tilde\lambda}(\Ci)$.
\item[(v)] $D(\tilde G)\subset D(\tilde F)\subset \Q_{\pi}$.
\item[(vi)] $\tilde G(D(\tilde G))\subset \tilde G(D(\tilde F))\subset\Q$; in particular $\tilde G(D(\tilde G))\cap D(\tilde G)=\emptyset$.
    \item[(vii)] The set $D=D(g)=\phi(D(\tilde G))$ of discontinuities of the derivative of $g$ satisfies $$g(D)\cap D=\emptyset.$$
        \item[(viii)] For every $x\in \Q_{\pi}$, $\tilde G(x)\in\Q$.
 \end{itemize}
 \end{lemma}
 \begin{proof}To verify (i), for $0\le u<v\le 1$ we can write with the help of (\ref{e:2})
 \begin{align*}&\lambda(G^{-1}((u,v)))=\lambda(F^{-1}(h^{-1}(u,v)))=\lambda(F^{-1}((h^{-1}(u),h^{-1}(v))))=\\
 &=\lambda(F^{-1}((0,h^{-1}(v))))-\lambda(F^{-1}((0,h^{-1}(u))))=v-u.
 \end{align*}

 (ii) This is because $h(x)=0$ if and only if $x=0$.

(iii) It follows from the fact that $G=h\circ F$ and outer map $h$ is an increasing continuous piecewise affine homeomorphism of $[0,1]$.

(iv) By the previous property (iii), $\tilde G\in\tilde \F(\R)$. Lemma \ref{l:3}(ii) and (i) furthermore imply that $\tilde G\in\tilde\F_{\lambda}(\R)$ and $g\in C_{\tilde\lambda}(\Ci)$.

(v) This is because the slopes of piecewise affine outer homeomorphism $h$ can change only at the points from $F(T(\tilde F))$.

(vi) For each interval $(u,v)$, where
\begin{itemize}
\item[(1)] either $u=0$ and $v$ is the least value $F(x)>0$ at a turning point $x$ of $\tilde F$,
\item[(2)] or $u,v$ are two consecutive values at turning points of $\tilde F$,
\item[(3)] or $u$ is the biggest value $F(x)<1$ at a turning point $x$ of $\tilde F$ and $v=1$,
\end{itemize}
$F^{-1}((u,v))$ can be expressed as a finite union
$$
F^{-1}((u,v))=\bigcup_{j}(a_j,b_j),\text{ where }F((a_j,b_j))=(u,v)\text{ for each }j.
$$
It follows from our definition of $\tilde \F_0(\R)$ in (\ref{e:41}) that $a_j,b_j\in\Q_{\pi}$ for all $j$, so
\begin{align}\label{a:4}\lambda(F^{-1}((u,v)))=\sum_j(b_j-a_j)\in\Q.
\end{align}
Fix a turning point $w\in T_{[0,1]}(\tilde F)$ for which $F(w)>0$. One can set
$$0=u_1<v_1=u_2<v_2=\cdots=v_k=F(w),
$$
where $u_i,v_i$ were described above in (1)-(3); then by (\ref{e:2}) and (\ref{a:4}),
\begin{align}\label{a:5}
G(w)=(h\circ F)(w)=&\lambda(F^{-1}(0,F(w)))=\sum_i\lambda(F^{-1}((u_i,v_i)))=\\
&=\sum_i\sum_j(b_j-a_j)\in\Q\nonumber.
\end{align}
By (v) and (\ref{e:41})(ii), $D(\tilde G)\subset D(\tilde F)\cap [0,1]= F^{-1}(S)\subset\Q_{\pi}$. Since $G(x)\in\Q$ if and only if $\tilde G(x)\in\Q$, from (\ref{a:5}) and (\ref{e:41})(ii) we obtain $\tilde G(D(\tilde G))=\tilde G(D(\tilde F))\subset\Q$ hence $$\tilde G(D(\tilde G))\cap D(\tilde G)=\emptyset.$$

(vii) This property is a consequence of (iv) and the fact that $\tilde G$ is a lifting of $g$ due to formula (\ref{e:10}).

To prove (viii), by conditions (v) and (vi) we can assume that $x\notin D(\tilde{F})$. Let $x\in (p,q)\cap \mathbb{Q}_{\pi}$, where $p,q\in D(\tilde F)$ are adjacent. Then
$$
\tilde G(x)=\tilde G(p)+\frac{\tilde G(q)-\tilde G(p)}{q-p}(x-p).
$$
By (vi), each of the numbers $\tilde G(p)$, $\tilde G(q)-\tilde G(p)$, $q-p$ and $x-p$ is rational, so $\tilde G(x)\in\Q$.

 \end{proof}

 Using Lemma \ref{l:31} we introduce the set $C_{\tilde \lambda,0}(\Ci)$ of the circle maps from $C_{\tilde \lambda}(\Ci)$ with liftings in $\tilde\F_{\lambda,0}(\R)$, where
\begin{equation}\label{e:20}\tilde\F_{\lambda,0}(\R):=\{\tilde G\colon~\tilde F\in \tilde\F_0(\R)\text{ and }G=h\circ F\}.
  \end{equation}

Recall that by our definition the set $\tilde\F(\R)$ consists of liftings of {\it onto} circle maps (see Remark \ref{r:1} and the text preceding it). 

\begin{lemma}\label{l:5}The set $C_{\tilde \lambda,0}(\Ci)$ is dense in $C_{\tilde \lambda}(\Ci)$.
\end{lemma}
\begin{proof}  Fix $\eps>0$ and a map $e\in C_{\tilde\lambda}(\Ci)$ with a lifting $\tilde E\in\tilde\F_{\lambda}(\R)$. By Lemma \ref{l:7} and Remark \ref{r:1} there is a map $\tilde F\in\tilde\F_0(\R)$ such that its representative $F=\tilde F\vert [0,1)$ is onto and
\begin{itemize}
\item[(i)]$\rho(\tilde E,\tilde F)<\frac{\eps}2$,
\item[(ii)]for $h\colon [0,1]\to [0,1]$ defined by $h(0)=0$, $h(x)=\lambda(F^{-1}((0,x)))$ (as in \eqref{e:2})
\begin{equation*}\label{e:22}\rho(h,\id)<\frac{\eps}2.\end{equation*}
\end{itemize}
Condition (ii) can be fulfilled due to the following reasoning. If circle maps $f_n$ converge in the uniform metric to a Lebesgue measure preserving circle map then their corresponding $h_n$ defined as in (ii) converge to $\text{id}$ - we refer the reader to \cite{Bo91} where the analogous interval case had been treated in details. We have proved in Lemma \ref{l:31}(iv) that $G=h\circ F$ is a representative of a map $g$ from $C_{\tilde\lambda,0}(\Ci)$. Moreover, using (i) and (ii) and the definition of $\tilde G$ in Lemma \ref{l:31}(iii),(iv) showing that $\rho(\tilde F,\tilde G)=\rho(F,G)$, we obtain
\begin{align}\label{a:6}
&\rho(e,g)\le \rho(\tilde E,\tilde G)\le \rho(\tilde E,\tilde F)+\rho(\tilde F,\tilde G)=\nonumber \\
&=\rho(\tilde E,\tilde F)+\rho(F,G)=\rho(\tilde E,\tilde F)+\rho(F,h\circ F)<\eps.\nonumber
\end{align}
Thus, for each $\eps>0$ and $e\in C_{\tilde\lambda}(\Ci)$ we have found a map $g\in C_{\tilde\lambda,0}(\Ci)$ such that $\rho(e,g)<\eps$.
\end{proof}

\begin{definition}
We say that two maps $f,g\colon~[a,b]\subset [0,1]\to \R$ are $\lambda$-equivalent if for each Borel set $A\subset \R$,
$$\lambda(f^{-1}(A))=\lambda(g^{-1}(A)).
$$
\end{definition}

For $0\le a<b\le 1$ and $0\le c<d\le 1$ we denote by $^+h_{{[a,b];[c,d]}},^-h_{[a,b];[c,d]}$ the affine maps from $[a,b]$ onto $[c,d]$ fulfilling $^+h(a)=c$, $^+h(b)=d$, resp.\ $^-h(a)=d$, $^-h(b)=c$.

\begin{lemma}\label{l:1}For any map $f\in C_{\lambda}(I)$, the maps $$^{\pm}h_{[0,1];[c,d]}\circ f\circ ^+h_{[a,b];[0,1]}$$
are $\lambda$-equivalent to the maps $^{\pm}h_{[a,b];[c,d]}$.
\end{lemma}
\begin{proof}For each Borel $A\subset [c,d]$ we have
\begin{align}&\lambda((^{\pm}h_{[0,1];[c,d]}\circ f\circ ^+h_{[a,b];[0,1]})^{-1}(A))=\nonumber\\
&=\lambda((^+h_{[a,b];[0,1})^{-1}\circ f^{-1}\circ (^{\pm}h_{[0,1];[c,d]})^{-1}(A))=\nonumber\\
&=\frac{b-a}{d-c}\lambda(A)=\lambda((^{\pm}h_{[a,b];[c,d]})^{-1}(A)).\nonumber\end{align}
\end{proof}

We will apply Lemma \ref{l:1} for two special classes of elements from $C_{\lambda}(I)$.  The first one consists of piecewise affine maps with $2n+1$, $n\in\N$, full laps: for points $0=x_0<x_1<\cdots<x_{2n}<x_{2n+1}=1$, $\bar{x}=(x_0,x_1,\dots,x_{2n+1})$ and $i\in\{0,1,\dots,2n+1\}$ we define $\beta[2n+1,\bar{x}]\in C_{\lambda}(I)$ as
\begin{equation}\label{e:4}
\beta[2n+1,\bar{x}](x_{i}) :=
\begin{cases}
0,~i\text{ even},\\
1,~i\text{ odd},
\end{cases}
\end{equation}
and continuous, affine on each $[x_i,x_{i+1}]$.

The second class that we define consists of maps $\Psi=\Psi[\eps,a',d,e,h']$, where
\begin{equation}\label{e:5}0<\frac{\eps}{3}<a',~a'+\frac{2\eps}{3}<d<e<h'-\frac{2\eps}{3},~h'<1-\frac{\eps}{3},
\end{equation}
as illustrated by the right part of Figure \ref{fig:dendritemap1} and its caption.

\begin{figure}[!ht]
	\centering
	\begin{tikzpicture}[scale=5]
	\draw (0,0)--(0,1)--(1,1)--(1,0)--(0,0);
	\draw[dashed] (1/20,0)--(1/20,1);
	\draw[dashed] (2/20,0)--(2/20,1);
	\draw[dashed] (3/20,0)--(3/20,1);
	\draw[dashed] (5/20,0)--(5/20,1);
	\draw[dashed] (0,1/6)--(1,1/6);
	\draw[dashed] (0,5/6)--(1,5/6);
	\draw[dashed] (1-1/20,0)--(1-1/20,1);
	\draw[dashed] (1-3/20,0)--(1-3/20,1);
	\draw[dashed] (13/20,0)--(13/20,5/6);
	\node at (0.5,0.5) {$\Phi$};
	\node at (0,-0.05) {\small $0$};
	\node at (1/20,-0.05) {\small  $a$};
	\node at (2/20,-0.05) {\small  $b$};
	\node at (3/20,-0.05) {\small  $c$};
	\node at (5/20,-0.05) {\small  $d$};
	\node at (13/20,-0.05) {\small $e$};
	\node at (1-1/20,-0.05) {\small $h$};
	\node at (1-2/20,-0.05) {\small $g$};
	\node at (1-3/20,-0.05) {\small $f$};
	\node at (1,-0.05) {\small  $1$};
	\node at (-0.05,1/6) {\small  $\eps$};
	\node at (-0.1,5/6) {\small  $1-\eps$};
	\draw[thick] (0,0)--(1/20,1/6);
	\draw[thick] (1/20,5/6)--(2/20,1)--(3/20,5/6)--(5/20,1/6)--(13/20,5/6)--(17/20,1/6)--(18/20,0)--(19/20,1/6);
	\draw[thick] (1,1)--(1-1/20,1-1/6);
	\end{tikzpicture}
	\hspace{0,1cm}
	\begin{tikzpicture}[scale=5]
	\draw (0,0)--(0,1)--(1,1)--(1,0)--(0,0);
	\draw[dashed] (1/20,0)--(1/20,1/6);
	\draw[dashed] (2/20,0)--(2/20,5/6);
	\draw[dashed] (3/20,0)--(3/20,1);
	\draw[dashed] (4/20,0)--(4/20,5/6);
	\draw[dashed] (5/20,0)--(5/20,1);
	\draw[dashed] (0,1/6)--(1,1/6);
	\draw[dashed] (0,5/6)--(1,5/6);
	\draw[dashed] (13/20,0)--(13/20,5/6);
	\draw[dashed] (1-1/20,0)--(1-1/20,5/6);
	\draw[dashed] (1-2/20,0)--(1-2/20,1/6);
	\draw[dashed] (16/20,0)--(16/20,1/6);
	\node at (0.5,0.5) {$\Psi$};
	\node at (0,-0.05) {\small $0$};
	\node at (1/20,-0.06) {\small  $a$};
	\node at (2/20,-0.05) {\small  $a'$};
	\node at (3/20,-0.05) {\small  $b'$};
	\node at (4/20,-0.05) {\small  $c'$};
	\node at (5/20,-0.05) {\small  $d$};
	\node at (13/20,-0.05) {\small $e$};
	\node at (1-4/20,-0.05) {\small $f'$};
	\node at (1-1/20,-0.05) {\small $h$};
	\node at (1-2/20,-0.05) {\small $h'$};
	\node at (1-3/20,-0.05) {\small $g'$};
	\node at (1,-0.05) {\small  $1$};
	\node at (-0.05,1/6) {\small  $\eps$};
	\node at (-0.1,5/6) {\small  $1-\eps$};
	\draw[thick] (0,0)--(1/20,1/6)--(2/20,5/6)--(3/20,1)--(4/20,5/6)--(5/20,1/6)--(13/20,5/6)--(16/20,1/6)--(17/20,0)--(18/20,1/6)--(19/20,5/6)--(1,1);
	\end{tikzpicture}
	\caption{For $\eps\in (0,1/2)$, $a=b-a=c-b=g-f=h-g=1-h=b'-a'=c'-b'=g'-f'=h'-g'=\frac{\eps}{3}$. }\label{fig:dendritemap1}
\end{figure}
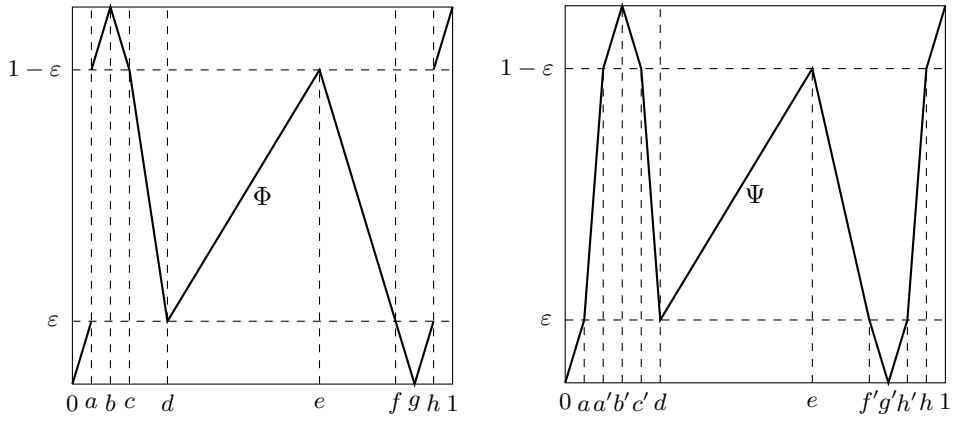

\begin{lemma}\label{l:2}For each choice of values $\eps,a',d,e,h'$ fulfilling (\ref{e:5}), $\Psi[\eps,a',d,e,h']\in C_{\lambda}(I)$.\end{lemma}
\begin{proof}
Consider the discontinuous map $\Phi$ defined by  the left part of Figure \ref{fig:dendritemap1}. The reader can easily verify that for each $y\in (0,1)\setminus\{\eps,1-\eps\}$,
\begin{equation}\label{e:1}
\sum_{x\in \Phi^{-1}(y)}\frac{1}{\vert \Phi'(x)\vert}=1.
\end{equation}
Now pick the points $a,b,b',c,c',f,f',g,g',h\in [0,1]$ as suggested in Figure~\ref{fig:dendritemap1}.
The map $\Psi$ from the right part of Figure~\ref{fig:dendritemap1} is continuous. Using its description in the caption of Figure~\ref{fig:dendritemap1} let us show that $\Psi\in C_{\lambda}(I)$. It is clear that for any $y\in (0,\eps)\cup (1-\eps,1)$ the equality (\ref{e:1}) estimated for $\Psi$ holds true again.
For any $y\in (\eps,1-\eps)$ we can write
\begin{align}\label{a:1}
&\sum_{x\in \Psi^{-1}(y)}\frac{1}{\vert \Psi'(x)\vert}=\frac{a'-a}{1-2\eps}+\frac{d-c'}{1-2\eps}+\frac{e-d}{1-2\eps}+\\
&+\frac{f'-e}{1-2\eps}+\frac{h-h'}{1-2\eps}=\clubsuit\nonumber,
\end{align}
since $a'-a+d-c'=d-c$ and $f'-e+h-h'=f-e$, we can rewrite (\ref{a:1}) with the help of (\ref{e:1}) as
\begin{align*}
&\clubsuit=\frac{d-c}{1-2\eps}+\frac{e-d}{1-2\eps}+\frac{f-e}{1-2\eps}=1,
\end{align*}
i.e., $\Psi\in C_{\lambda}(I)$ by Lemma \ref{l:6}.
\end{proof}

\subsection{Partitions, special perturbations}\label{sec:PartitionsSpecialPerturbations}

In this section we will start with maps from $C_{\tilde \lambda,0}(\Ci)$ defined in  the previous section and particular associated partitions of $\Ci$ and show how to perturb such maps and refine their associated partitions so that they will satisfy conditions (C1)-(C6) given below. This will provide us with the crucial step in proving genericity of s-limit shadowing in the next section. For what follows we refer the reader to see Figure~\ref{fig:dendritemap2} to visualize the discussed concepts better.

Given a piecewise affine circle map $g\in C_{\tilde \lambda,0}(\Ci)$, $\eps>0$ and its \emph{affine
partition} $\PP\supset D(g)$ for which  $$\PP\subset \phi(\Q_{\pi})\text{ and }\vert\vert\PP\vert\vert<\eps,$$ (where $\vert\vert \cdot \vert\vert$ denotes the maximum diameter of partition elements) we will construct a
perturbation $\theta$ of $g$ and a partition $\PQ\subset \phi(\Q_{\pi})$ for $\theta$ for which $\PP\prec\PQ$ (i.e. $\PQ$ \emph{refines} $\PP$)
and such that each $J\in \PQ$ has a subdivision into subarcs $L^J_1, L^J_2,M^J,R^J_2,R^J_1$ whose order preserves order in $J$ satisfies

\begin{enumerate}[(C1)]
	\item\label{C1} There is $I\in \PQ$ (depending on $J$) such that $\theta(J)\supset I$,
	\item\label{C2} Let $I\in \PQ$ be such that $\theta(J)\cap (L^I_1\cup L^I_2)\neq \emptyset$. Then
	\begin{enumerate}
		\item\label{C21} $L^I_1\cup L^I_2\subset  \theta(J)$,
		\item\label{C22} if $K\in \PQ$ is the unique element such that $R_1^K\cap L_1^I\neq \emptyset$
		then $R^K_1\cup R^K_2\subset  \theta(J)$.
	\end{enumerate}
	\item\label{C3} Let $I\in \PQ$ be such that $\theta(J)\cap (R^I_1\cup R^I_2)\neq \emptyset$. Then
	\begin{enumerate}
		\item\label{C31} $R^I_1\cup R^I_2\subset  \theta(J)$,
		\item\label{C32} if $K\in \PQ$ is the unique element such that $L_1^K\cap R_1^I\neq \emptyset$
		then $L^K_1\cup L^K_2\subset  \theta(J)$.
	\end{enumerate}
	\item\label{C4} $\theta(J)=\theta(L_1^J)=\theta(R_1^J)$,
	\item\label{C5} $B_{4\eta}(\theta(M^J\cup L_2^J\cup R_2^J))\subset \theta(J)$ for sufficiently small $\eta>0$.
\end{enumerate}

By (\ref{e:20}), the map $g$ has its lifting $\tilde G$ from $\F_{\tilde\lambda,0}(\R)$ represented by $G=h\circ F\in\F([0,1))$, where $F$ and $h$ were described immediately prior to Lemma \ref{l:31}. By (iii),(iv) of Lemma \ref{l:31}, $g$ is piecewise affine, i.e., such that the map $\tilde G$, resp.\ $G$ is piecewise affine. Applying (\ref{e:2}) and Lemma \ref{l:31} we can consider a finite set $P$ of points such that
$$D(\tilde F)\cap [0,1]\subset P:=\{0<p_1<p_2<\cdots<p_m<1\}\subset \Q_{\pi}$$ for which $\tilde G\vert [p_i,p_{i+1}]$ is affine for each $i$ (set $p_{m+1}=p_1+1$), and for \begin{equation*}\label{e:21}\PP:=\phi(P)=\{\phi(p_1),\dots,\phi(p_m)\}
\subset\phi(\Q_{\pi})\subset \Ci,\end{equation*}
\begin{itemize}
\item[($a$)] $g(\PP)\subset \phi(\Q)$ hence $g(\PP)\cap \PP=\emptyset$,
\item[($b$)] $\vert\vert\PP\vert\vert:=\max_{1\le i\le m}\vert \phi(p_{i+1})-\phi(p_i)\vert<\eps$.
\end{itemize}

We will call the set $P$, resp.\ $\PP$ a partition for $\tilde G$, resp.\ $g$. Redefining $\tilde G$ on each $[p_i,p_{i+1}]$ by  (the numbers $n(i)\in\N$ and vector $\bar{x}(i)$ will be specified later)

\begin{equation}\label{e:7}\tilde \Sigma_i:=^{s(i)}h_{[0,1];[\tilde G(p_i),\tilde G(p_{i+1})]}\circ \beta[2n(i)+1,\bar{x}(i)]\circ ^{+}h_{[p_i,p_{i+1}];[0,1]},\end{equation}
where $\beta$'s were introduced in (\ref{e:4}) and $s(i)\in\{+,-\}$ ares chosen to satisfy $\tilde \Sigma_i(p_i)=\tilde G(p_i)$, yields a map $\tilde \Sigma\colon~[0,1]\to\R$ given by $\tilde \Sigma(x)=\tilde \Sigma_i(x)$, $x\in [\tilde{p}_i,\tilde{p}_{i+1}]$. Notice that still $\tilde\Sigma(1)-\tilde \Sigma(0)=\degr(g)$, so abusing the notation we will again denote by $\tilde \Sigma$ its extension from $[0,1]$ to the whole real line keeping the rule $\tilde \Sigma(x+1)=\tilde \Sigma(x)+\degr(g)$. In fact the map $\tilde \Sigma$ is a lifting of some map $\sigma\colon~\Ci\to\Ci$. Because by Lemma \ref{l:1} each map $\tilde G\vert [p_i,p_{i+1}]$ has been replaced by a $\lambda$-equivalent map $\tilde \Sigma_i$, it follows that the map $\Sigma\in\F([0,1))$ representing $\sigma$ satisfies the conditions of Lemma~\ref{l:6} hence by Lemma~\ref{l:3} (i) it holds that $\sigma\in C_{\tilde\lambda}(\Ci)$. For the map $\tilde \Sigma\vert [0,1]$ we will consider a new partition
\begin{equation*}Q:=\bigcup_{i=1}^m\bigcup_{j=0}^{2n(i)+1} h^{-1}_{[p_i,p_{i+1}];[0,1]}(x_j(i))=:
\{0=q_1<q_2<\cdots<q_{m'}=1\},
\end{equation*}
for some $m'\in \N$, where the vectors $\bar{x}(i)=(x_0(i),x_1(i),\dots,x_{2n(i)+1}(i))$ will be chosen to satisfy $Q\subset\Q_{\pi}$. Thus, the set $Q$ contains $P$ and also all new turning points of $\tilde \Sigma$ in $(0,1)$ being in $Q\setminus P$.  From our specific choice of $\beta$'s in (\ref{e:7}) and Lemma \ref{l:31}(viii) we obtain
$$Q\subset\Q_{\pi}\text{ and }\tilde \Sigma(Q)\subset \Q;$$ denoting $\PQ=\phi(Q)$ we analogously obtain for $\sigma$ and $\PQ$
$$
\PQ\subset\phi(\Q_{\pi})\text{ and }\sigma(\PQ)\subset \phi(\Q);$$
which implies that
$$\sigma(\PQ)\cap \PQ=\emptyset.
$$

At the same time the numbers $n(i)$ (recall that the number of full laps of $\beta$ is $2n(i)+1$) can be taken sufficiently large to satisfy for each $i$ and arcs $[\phi(q_i),\phi(q_{i+1})]=\phi([q_i,q_{i+1}])$,
 \begin{equation*}\label{e:9}\#\{j\colon~\sigma([\phi(q_i),\phi(q_{i+1})])
 \cap [\phi(q_j),\phi(q_{j+1})]\neq\emptyset\}\ge 3.
 \end{equation*}

Up to now, using rescaled versions of $\beta$'s  we have perturbated the map $\tilde G$ (resp.\ $g$) on the intervals $[p_i,p_{i+1}]$ (resp.\ arcs $\phi([p_i,p_{i+1}])$) to obtain the lifting $\tilde \Sigma$ of $\sigma\in C_{\tilde\lambda}(\Ci)$.

In the last part of this proof we will proceed similarly: using rescaled versions of $\Psi$'s  from Figure~\ref{fig:dendritemap1} we will perturb the map $\tilde \Sigma$ (resp.\ $\sigma$) on the intervals $[q_i,q_{i+1}]$ (resp.\ arcs $\phi([q_i,q_{i+1}])$) to obtain the lifting $\tilde \Te$ of $\ta\in C_{\tilde\lambda}(\Ci)$.

Therefore, for each $i$ define
\begin{equation}\label{e:6}\tilde \Theta_i:=^{s(i)}h_{[0,1];[\tilde \Sigma(q_i),\tilde\Sigma (q_{i+1})]}\circ \Psi[\eps_i,a'_i,d_i,e_i,h'_i]\circ ^{+}h_{[q_i,q_{i+1}];[0,1]},
\end{equation}
where $s(i)\in \{+,-\}$ is chosen to satisfy
\begin{equation}\label{e:8}\tilde \Theta_i(q_i)=\tilde \Sigma(q_i);
\end{equation}
using $\tilde \Theta_i$, we can define the map $\tilde \Theta\colon~[0,1]\to \R$ by 
$\tilde \Theta(x)=\tilde \Theta_i(x)$, $x\in [q_i,q_{i+1}]$. The reason why the degree preserving extension of $\tilde \Theta$ to the real 
line is a lifting of a map $\theta\in C_{\tilde\lambda}(\Ci)$ is analogous as above: the map $\theta$ is represented by the map 
$\Theta=\tilde \Theta\vert [0,1)(\mo~ 1)\in\F([0,1))$ that fulfills conditions of Lemma~\ref{l:6}. Let us consider the map $\tilde \Theta$, resp.\ $\theta$ with 
respect to partition $Q$, resp.\ $\PQ$. For what follows we refer the reader to the right picture in Figure~\ref{fig:dendritemap1}. Taking in (\ref{e:6}) $\eps_i$, $a'_i$, $d_i$, $e_i$ and $h'_i$ such that 
\begin{equation}\label{e:23}h^{-1}_{[q_i,q_{i+1}];[0,1]}(\{a_i,a'_i,b'_i,c'_i,d_i,e_i,f'_i,g'_i,h'_i,h_i\})\subset\Q_{\pi},\end{equation} $\eps_i$ and $d_i$ sufficiently close to $0$ and $e_i$ sufficiently close to $1$, with the help of (\ref{e:8}) we can ensure that for each $i$,
\begin{align*}\label{a:3}&\tilde \Sigma(Q)\cap [q_i,q_{i+1}]=\tilde \Theta(Q)\cap [q_i,q_{i+1}]=\\
&=\tilde\Theta(Q)\cap (h^{-1}_{[q_i,q_{i+1}];[0,1]}(d_i),
h^{-1}_{[q_i,q_{i+1}];[0,1]}(e_i)).
\end{align*}

Let us put for each $i$, $[q_i,q_{i+1}]$ and $H:=h_{[q_i,q_{i+1}];[0,1]}$, \begin{align*}&L^i_1=H^{-1}([0,c'_i]),~L^i_2=H^{-1}([c'_i,d'_i]),~ M^i=H^{-1}
([d'_i,e'_i]),\\
&R^i_2=H^{-1}([e'_i,f'_i]),~R^i_1=H^{-1}([f'_i,1]);
\end{align*}
then using $\phi$ we can transfer these sets to the arc $$J=\phi([q_i,q_{i+1}])$$
by
\begin{align*}&L^J_1=\phi(L^i_1),~L^J_2=\phi(L^i_2),~ M^J=
\phi(M^i),\\
&R^J_2=\phi(R^i_2),~R^J_1=\phi(R^i_1).
\end{align*}
The sketch of this construction is drawn in  Figure~\ref{fig:dendritemap2}.

Since $\PQ\subset\phi(\Q_{\pi})$ is a refinement of $\PP$, from $(b)$ we obtain that
$\vert\vert\PQ\vert\vert<\eps$. By Lemma \ref{l:31} $\theta(\PQ)\cap\PQ=\emptyset$ and
the conditions (C1)-(C5) for the map $\theta\in C_{\tilde\lambda}(\Ci)$ with respect to $$\PQ=\PQ_{\eps,\theta}:=\phi(Q)$$ easily follow.

\begin{figure}
	\begin{tikzpicture}[scale=5]
	\draw (0,0)--(0,1)--(0.8,1)--(0.8,0)--(0,0);
	\node at (0.3,0.5) {\bf $\theta$};
	\node at (0.1,-0.05) {\small\bf $L^J_1$};
	\node at (0.2,-0.05) {\small\bf $L^J_2$};
	\node at (0.4,-0.05) {\small\bf $M^J$};
	\node at (0.6,-0.05) {\small\bf $R^J_2$};
	\node at (0.73,-0.05) {\small\bf $R^J_1$};
	\draw[dashed](0,1/6)--(0.8,1/6);
	\draw[dashed](0,5/6)--(0.8,5/6);
	\draw[dashed](0,1/6)--(0.8,1/6);
	\draw[dashed](4/20*0.8,0)--(4/20*0.8,5/6);
	\draw[dashed](5/20*0.8,0)--(5/20*0.8,1/6);
	\draw[dashed](13/20*0.8,0)--(13/20*0.8,5/6);
	\draw[dashed](16/20*0.8,0)--(16/20*0.8,1/6);
	\draw[thick](0,0)--(1/20*0.8,1/6)--(2/20*0.8,5/6)--(3/20*0.8,1)--(4/20*0.8,5/6)--(5/20*0.8,1/6)--(13/20*0.8,5/6)--(16/20*0.8,1/6)--(17/20*0.8,0)--(18/20*0.8,1/6)--(19/20*0.8,5/6)--(0.8,1);
	\end{tikzpicture}
	\caption{$J=\phi([q_i,q_{i+1}])=L^J_1\cup L^J_2\cup M^J\cup R^J_2\cup R^J_1$. }\label{fig:dendritemap2}
\end{figure}
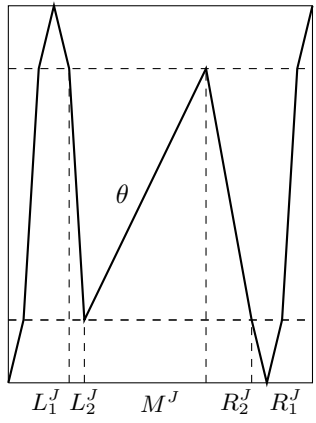

\subsection{S-limit shadowing is generic in $C_{\tilde\lambda}(\Ci)$}\label{sec:MainProof}

For a given $\eps>0$, assume that $\theta$, $\PQ_{\eps,\theta}$ with $\vert\vert\PQ_{\eps,\theta}\vert\vert<\eps$ and $\eta$ are provided in such a way that they satisfy conditions (C1)-(C5) and there is also $\delta=\delta(\theta)>0$ such that:

\begin{enumerate}[(C1)]
   \setcounter{enumi}{5}
	\item\label{C6} $\delta<\eta$ and $2\delta<\diam K$ for any $K\in \{L^J_1,L^J_2,M^J, R_1^J, R^J_2\}$ and any $J\in \PQ_{\eps,\theta}$.
\end{enumerate}

\begin{lem}\label{l:4}Let $\eps>0$, $\theta$, $\PQ_{\eps,\theta}$ and $\delta=\delta(\theta)>0$ be as above. For every $\tau\in C_{\tilde\lambda}(\Ci)$ such that $\rho(\tau,\theta)<\delta$ every $\delta$-pseudo orbit for $\tau$ is $\eps$-traced.
\end{lem}
\begin{proof}
Let $\mathbf{x}=\set{x_s}_{s=0}^\infty$ be a $\delta$-pseudo orbit for $\tau$. We claim that there is a sequence of arcs $J_s\in \PQ=\PQ_{\eps,\theta}$ and sets $Q_s\subset J_s$ such that
\begin{enumerate}
	\item\label{1} $x_s\in J_s$,
	\item\label{2} $\tau(Q_s)\supset Q_{s+1}$ and $Q_s\in \{L_1^{J_s},R_1^{J_s}\}$.
\end{enumerate}
As $J_0\in \PQ$ select any arc such that $x_0\in J_0$ (in the worst case there are two such arcs). Fix any $Q_0\in \{L_1^{J_0},R_1^{J_0}\}$.

Now suppose that the above conditions are satisfied for some $s$
and let $J_{s+1}\in \PQ$ be such that $x_{s+1}\in J_{s+1}$. If $\theta(x_s)\in J_{s+1}$ then since $\theta(J_s)$ contains at least one element of $\PQ$, by condition (C\ref{C1}) we have that $\theta(J_s)\cap L_1^{J_{s+1}}\neq \emptyset$ or $\theta(J_s)\cap R_1^{J_{s+1}}\neq \emptyset$. In the first case $L^{J_{s+1}}_1\cup L^{J_{s+1}}_2\subset  \theta(J_s)$ and
$R^{J_{s+1}}_1\cup R^{J_{s+1}}_2\subset  \theta(J_s)$ in the second case. Also $\theta(J_s)=\theta(Q_s)$. But then, by the definition of $\delta$ and the conditions (C\ref{C2}), (C\ref{C3}), (C\ref{C5}) we have either $L^{J_{s+1}}_1\subset \tau(Q_s)$ or
$R^{J_{s+1}}_1\subset \tau(Q_s)$. Put $Q_{s+1}=L^{J_{s+1}}_1$ and $Q_{s+1}=R^{J_{s+1}}_1$ in respective cases and observe that the claim holds.
But then, since $\vert\vert\PQ\vert\vert<\eps$, it is enough to choose $z\in \cap \tau^{-s}(Q_s)$ to obtain a point $\eps$-tracing $\mathbf{x}$.
\end{proof}

In the previous section we have described a special type of perturbation of a piecewise affine map $g\in C_{\tilde\lambda,0}(\Ci)\subset C_{\tilde\lambda}(\Ci)$ resulting with a circle map $\theta$. The main property of $\theta$ was stated in Lemma \ref{l:4}. Now we are going to apply similar approach to a dense sequence of piecewise affine maps from $C_{\tilde\lambda,0}(\Ci)$ which is possible by invoking Lemma \ref{l:5}.

To that end, let $\Gamma:=\{g_m\}_{m\ge 1}\subset C_{\tilde \lambda,0}(\Ci)$ be a dense sequence of maps in $C_{\tilde\lambda}(\Ci)$ such that
\begin{itemize}
\item each $g_m$ has an affine partition $\PP_m\subset\phi(\Q_{\pi})$ satisfying $\vert\vert\PP_m\vert\vert<\frac{1}{m}$,
    \item for each $n\ge m$, $g_n(\PP_m)\cap \PP_m=\emptyset$.
    \end{itemize}
    Notice that the second property is guaranteed by Lemma \ref{l:31}(viii). Following the previous section we perturb $g_m$ to $\theta_m$ with $\eps=\frac{1}{m}$, corresponding partition $\PQ_m:=\PQ_{\frac{1}{m},\theta_m}$, $\eta_m$, $\delta_m=\delta(\theta_m)<\frac{1}{m}$ and $U_m$ an open neighborhood around $\theta_m$ in $C_{\tilde\lambda}(\Ci)$ such that
    \begin{itemize}
\item $U_m\subset B_{\delta_{m}}(\theta_m)$,
    \item the boundary of $U_m$ does not intersect $\Gamma$.
    \end{itemize}
    We will proceed as follows to  construct sequences $\{\PQ_m\}_{m=1}^{\infty}$, where each $\PQ_m$ is a subset of $\phi(\Q_{\pi})$, and  $\{U_m\}_{m=1}^{\infty}$:  we repeatedly use Lemma \ref{l:31}(viii) and Lemma \ref{l:5}
    \begin{enumerate}
    \item we perturb $g_1$ to $\theta_1$ to obtain $g_n(\PQ_1)\cap \PQ_1=\emptyset$ for each $n\ge 1$,
        \item for $m>1$, having already constructed the sets $\PQ_i$ and $U_i$, $i=1,\dots,m-1$, in order to construct $\PQ_m$ and $U_m$ we distinguish two possibilities:
        \begin{enumerate}
        \item either $g_m\notin \bigcup_{i=1}^{m-1}U_i$ and then we construct $\PQ_m$ and $U_m$ to fulfill $$g_n(\PQ_m)\cap \PQ_m=\emptyset\text{ for each }n\ge m,~\overline{U}_m\cap \bigcup_{i=1}^{m-1}\overline{U}_i=\emptyset,$$
            \item or $g_m\in U_i$, where $i\le m-1$ is the largest number with this property; denoting $E(\PQ_i)$ the set of $\phi$-images of points defined in (\ref{e:23}) for all $J\in\PQ_i$,
            a new partition $\PQ_m\subset\phi(Q_{\pi})$ will fulfill $$E(\PQ_i)\cup\PQ_i\cup\PP_m\prec \PQ_m\text{ and }\overline{U}_m\subset U_i\setminus\bigcup_{i<j\le m-1}\overline{U}_j.$$
            In particular, $\PQ_m$ is an affine partition for $\theta_m$ which is a refinement of $\PQ_i$ and $g_n(\PQ_m)\cap \PQ_m=\emptyset\text{ for each }n\ge m$.
                \end{enumerate}
 In addition we require that the boundary of ${U}_m$ does not intersect $\Gamma$; this is possible since $\Gamma$ is countable.
\end{enumerate}
Let us put $A_n=\bigcup_{m\ge n}U_m$. Clearly, each $A_n$ is open and dense so the intersection $$A=\bigcap_{n\ge 1}A_n=\bigcap_{n\ge 1}\bigcup_{m\ge n}U_m$$ is a dense $G_{\delta}$ set in $C_{\tilde\lambda}(\Ci)$.

\begin{proof}[Proof of Theorem~\ref{thm:main}]
We will prove that each $\tau\in A$ has the s-limit shadowing property.
By our definition of $A$, there is an increasing sequence $\{m(k)\}_{k=1}^{\infty}$ such that
\begin{equation*}U_{m(1)}\supset U_{m(2)}\supset\cdots,~\{\tau\}=\bigcap_{k=1}^{\infty}U_{m(k)}.
\end{equation*}
Let $\mathbf{x}=\set{x_s}_{s=0}^\infty$ be an asymptotic pseudo orbit. By Lemma~\ref{l:4} the partition $\PQ_{m(i)}$ and $\delta=\delta_{m(i)}>0$ were chosen for $\alpha$-shadowing with $\alpha=1/m(i)$. Fix $k$ so that \begin{equation}\label{e:15}4/k<\delta.\end{equation} Let the partition $\PQ_{m(j)}$ and $\gamma:=\delta_{m(j)}>0$ where $i<j$ be provided for $\beta$-shadowing with $\beta=1/m(j)\le 1/k$. Note that by condition (C\ref{C6}) $\gamma<1/k$, hence together with (\ref{e:15}) we obtain \begin{equation}\label{e:16}\gamma+2/k<3/k<4/k<\delta.\end{equation}
Assume for simplicity that $\mathbf{x}$ is a $\delta$-pseudo orbit and it is a $\gamma$-pseudo orbit for all $s\geq N-1$ for some $N$.
Let $J_s\in \PQ_{m(i)}$ and $Q_s\subset J_s$ be provided as in (1),(2) of the proof of Lemma \ref{l:4} for $\mathbf{x}$ and by the same conditions, let $R_s\in \PQ_{m(j)}$, and let $W_s\subset R_s$ for $s\geq N-1$ be provided by the fact that $\tau\in U_{m(j)}\subset B_{\delta_{m(j)}}(\theta_{m(j)})$. In particular,
\begin{equation}\label{e:19}\tau(Q_s)\supset Q_{s+1},~s\ge 0,~\tau(W_s)\supset W_{s+1},~s\ge N-1.
\end{equation}

First, if $W_N\subset \tau(Q_{N-1})$ then we can switch directly from the arc $Q_{N-1}$ used for $\alpha$-tracing to the arc $W_N$ used for $\beta$-tracing.\\
Now, assume that $W_N\setminus \tau(Q_{N-1})\neq \emptyset$.  Notice that $W_N\cap \Int (Q_N)=\emptyset$, since otherwise
$$W_N\subset Q_N\subset \tau(Q_{N-1})\text{ because }\PQ_{m(i)}\prec\PQ_{m(j)},$$
which gives a contradiction. On the other hand, $x_{N-1}\in J_{N-1}$ and $d(x_N,\tau(x_{N-1}))<\gamma$ and $\diam (R_N)<1/k$. Also $W_N\subset R_N$ and $x_N\in R_N$. Then if $\xi:=\gamma+1/k$, $$W_N\subset R_N\subset B_{\xi} (\tau(J_{N-1}))\subset B_{\delta} (\tau(J_{N-1})),$$ where the last inclusion is a consequence of (\ref{e:16}). 
Since $W_N$ is not included in the arc $\tau(Q_{N-1})$ and both diameters $\diam L_2^{J_N}$ and $\diam R_2^{J_N}$ are greater than $2\delta$, the only possibility is that $$W_N\subset L_2^{J_{N}}\cup M^{J_{N}}\cup R_2^{J_{N}}.$$

But by (C\ref{C5}) we have $$B_{4\delta}(\theta_{m(i)}(W_N))\subset \theta_{m(i)}(J_N)=\theta_{m(i)}(Q_N),$$ and thus since $\tau\in U_{m(i)}$,
$$B_{3\delta}(\theta_{m(i)}(W_N))\subset \tau(Q_N)$$
and
\begin{equation}\label{e:17}B_{2\delta}(\tau(W_N))\subset \tau(Q_N).\end{equation}
On the other hand, from (\ref{e:19}) we obtain
\begin{equation}\label{e:18}W_{N+1}\subset \tau(W_N)\subset B_{2\delta}(\tau(W_N)).
\end{equation}

 Gluing (\ref{e:17}) and (\ref{e:18}) together, we get

 $$W_{N+1}\subset \tau(Q_N).$$ This allows us to switch from the arc $Q_{N}$ used for $\alpha$-tracing to the arc $W_{N+1}$ used for $\beta$-tracing.

Then using inductively the above construction we obtain that for every $\eps>0$, every $\tau\in \bigcap_{n\ge 1} A_n$ and every asymptotic pseudo orbit  $\mathbf{x}=\set{x_s}_{s=0}^\infty$ which is $\delta$-pseudo orbit, we can find a sequence of closed arcs $I_s\subset \Ci$ with the following properties:
\begin{enumerate}
	\item $\tau(I_s)\supset I_{s+1} $,
	\item $\diam (I_s\cup \{x_s\})<\eps$,
	\item for every $\beta>0$ there is $N>0$ such that $\diam (I_s\cup \{x_s\})<\beta$ for all $s>N$.
\end{enumerate}
Now it is enough to take any $z\in \bigcap_{s\ge 0} \tau^{-s}(I_s)$ to $\eps$-trace and asymptotically trace $\mathbf{x}$.
\end{proof}

\section{Final remarks}\label{sec:Final}

As we mentioned before, some inspiration for this paper comes from \cite{MazOpr} where it is proved that on manifolds (including dimension one) s-limit shadowing is dense in the class of continuous maps.
In particular, it is dense in continuous maps on the circle and the interval. It was also proven in our recent paper \cite{BCOT} that s-limit shadowing is dense also in Lebesgue measure preserving maps on the interval. Then, in the view of the above results and the results in the present paper it is natural to expect that s-limit shadowing is generic also in Lebesgue measure preserving interval maps.
Unfortunately, the proof of Theorem~\ref{thm:main} will not directly work in that case as we explain below.
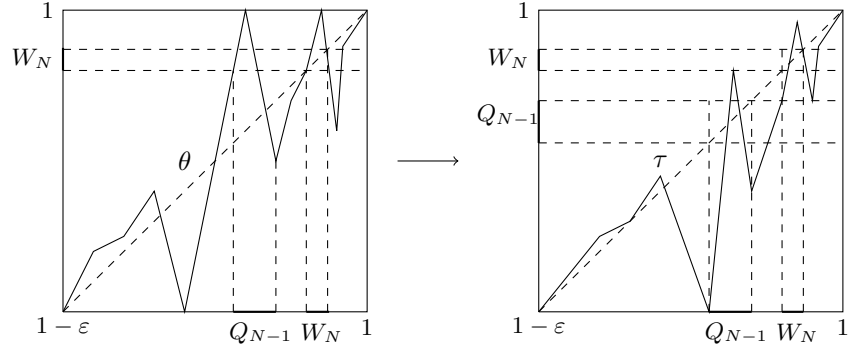
\begin{figure}[!ht]
	\centering
	\begin{tikzpicture}[scale=4]
	\draw (0,0)--(0,1)--(1,1)--(1,0)--(0,0);
	\draw[dashed] (0,0)--(1,1);
	\draw(0,0)--(0.1,0.2)--(0.2,0.25)--(0.3,0.4)--(0.4,0)--(0.6,1)--(0.7,0.5)--(0.75,0.7)--(0.8,0.8)--(0.85,1)--(0.9,0.6)--(0.92,0.88)--(1,1);
	\draw[thick] (0.56,0)--(0.7,0);
	\draw[thick] (0.8,0)--(0.875,0);
	\draw[thick] (0,0.8)--(0,0.875);
	\draw[dashed] (0.7,0)--(0.7,0.5);
	\draw[dashed] (0.8,0)--(0.8,0.8);
	\draw[dashed] (0.87,0)--(0.87,0.87);
	\draw[dashed] (0,0.8)--(1,0.8);
	\draw[dashed] (0,0.87)--(1,0.87);
	\draw[dashed] (0.56,0)--(0.56,0.8);
	\node at (-0.05,1) {\small $1$};
	\node at (0,-0.05) {\small $1-\eps$};
	\node at (0.65,-0.07) {\small $Q_{N-1}$};
	\node at (0.85,-0.07) {\small $W_N$};
	\node at (-0.1, 0.84) {\small $W_N$};
	\node at (1,-0.05) {\small  $1$};
	\node at (0.4,0.5) { $\theta$};
	\draw[->] (1.1,1/2)--(1.3,1/2);
	\end{tikzpicture}
	\begin{tikzpicture}[scale=4]
	\draw (0,0)--(0,1)--(1,1)--(1,0)--(0,0);
	\draw[dashed] (0,0)--(1,1);
	\draw(0,0)--(0.2,0.25)--(0.3,0.3)--(0.4,0.45)--(0.56,0)--(0.64,0.8)--(0.7,0.4)--(0.8,0.7)--(0.85,0.96)--(0.9,0.7)--(0.92,0.88)--(1,1);
	\draw[thick] (0.56,0)--(0.7,0);
	\draw[thick] (0.8,0)--(0.87,0);
	\draw[thick] (0,0.8)--(0,0.87);
	\draw[dashed] (0.7,0)--(0.7,0.7);
	\draw[dashed] (0.8,0)--(0.8,0.87);
	\draw[dashed] (0.87,0)--(0.87,0.87);
	\draw[dashed] (0,0.8)--(1,0.8);
	\draw[dashed] (0,0.87)--(1,0.87);
	\draw[dashed] (0.56,0)--(0.56,0.7);
	\draw[dashed] (0,0.56)--(1,0.56);
	\draw[dashed] (0,0.7)--(1,0.7);
	\draw[thick] (0,0.56)--(0,0.7);
	\node at (-0.05,1) {\small $1$};
	\node at (0,-0.05) {\small $1-\eps$};
	\node at (0.65,-0.07) {\small $Q_{N-1}$};
	\node at (-0.1,0.65) {\small $Q_{N-1}$};
	\node at (0.85,-0.07) {\small $W_N$};
	\node at (-0.1, 0.84) {\small $W_N$};
	\node at (1,-0.05) {\small  $1$};
	\node at (0.4,0.5) { $\tau$};
	
	\end{tikzpicture}
	\caption{After perturbation the image of $Q_{N-1}=Q_N$ covers itself. Therefore, $\eps$-tracing is still possible, however the image of $Q_{N-1}=Q_N$ does not cover $W_{N}=W_{N+1}$ anymore.}\label{fig:explanation}
\end{figure}
The main technique in our proof is showing that $W_N\subset \tau(Q_{N-1})$ or $W_{N+1}\subset \tau(Q_{N})$ under the map $\tau$ which is small perturbation of $\theta$, see the discussion after \eqref{e:19}, for more details. While for small perturbation we may ensure that $\tau(Q_{N-1})\supset Q_N$, we cannot control covering of smaller sets $W_s$ by $Q_s$, see Figure~\ref{fig:explanation} for an intuitive explanation of possible problems.
This situation may happen near endpoints of the interval, where we cannot guarantee sufficiently long overlapping of $\tau(Q_{N-1})$ or $\tau(Q_{N})$. Such a situation does not happen on the circle due to the lack of boundary. This motivates us to state the following question.

\begin{question}\label{q:B}
Is s-limit shadowing generic in Lebesgue measure preserving maps on the interval?
\end{question}

As we explained above, possible positive answer to the above question will require some new techniques, beyond the ones used in the present work.
On the other hand, a standard technique to disprove that a condition is generic is to find an open set without the property. Such approach is again impossible, because we have proven \cite{BCOT} that s-limit shadowing is dense in $C_{\lambda}(\mathbb{S}^1)$. 

\section*{Acknowledgements}
J. Bobok was supported by the European Regional Development Fund, project No.~CZ 02.1.01/0.0/0.0/16\_019/0000778. J. \v Cin\v c was supported by FWF Schrödinger Fellowship stand-alone project J 4276-N35. P. Oprocha was supported by National Science Centre, Poland (NCN), grant no. 2019/35/B/ST1/02239.

\end{document}